\documentclass[reqno]{amsart}

\usepackage{amssymb}
\usepackage{amsthm}
\usepackage{amsmath}

\newcommand{\C}{\mathbb{C}}
\newcommand{\D}{\mathbb{D}}
\newcommand{\N}{\mathbb{N}}

\newcommand{\T}{\mathbb{T}}
\newcommand{\Z}{\mathbb{Z}}

\newcommand{\p}{\mathbf{p}}

\newcommand{\mcF}{\mathcal{F}}
\newcommand{\barp}{\overline{p}}

\theoremstyle{plain}
\newtheorem{theorem}{Theorem}
\newtheorem{lemma}[theorem]{Lemma}
\newtheorem{proposition}[theorem]{Proposition}
\newtheorem{corollary}[theorem]{Corollary}

\theoremstyle{remark}
\newtheorem{remark}{\bf Remark}
\newtheorem{exmp}[remark]{\bf Example}

\title[\null]{Closed-form expression for finite predictor coefficients of multivariate ARMA processes}

\author[\null]{Akihiko Inoue}

\address{A.\ Inoue\\
Department of Mathematics\\
Hiroshima University\\
Higashi-Hiroshima 739-8526\\
Japan}
\email{inoue100@hiroshima-u.ac.jp}

\begin{document}



\begin{abstract}
We derive a closed-form expression for the finite predictor coefficients of multivariate ARMA (autoregressive moving-average) processes. The expression is given in terms of several explicit matrices that are of fixed sizes independent of the number of observations. The significance of the expression is that it provides us with a linear-time algorithm to compute the finite predictor coefficients. In the proof of the expression, a correspondence result between two relevant matrix-valued outer functions plays a key role. We apply the expression to determine the asymptotic behavior of a sum that appears in the autoregressive model fitting and the autoregressive sieve bootstrap. The results are new even for univariate ARMA processes.
\end{abstract}

\maketitle

\section{Introduction}\label{sec:1}

Let $\T:=\{z\in\C :\vert z\vert=1\}$ and $\overline{\D}:=\{z\in\C :\vert z\vert\le 1\}$ be 
the unit circle and the closed unit disk, in $\C$, respectively. 
For $d\in\N$, a $d$-variate ARMA (autoregressive moving-average) process 
$\{X_k : k\in\Z\}$ is a $\C^d$-valued, centered, weakly stationary process with spectral density 
$w$ of the form
\begin{equation}
w(e^{i\theta})=h(e^{i\theta})h(e^{i\theta})^*,\qquad \theta\in [-\pi,\pi),
\label{eq:farima529}
\end{equation}
where $h:\T\to \C^{d\times d}$ satisfies the following condition:
\begin{equation}
\mbox{the entries of $h(z)$ are rational functions in $z$ that have 
no poles in $\overline{\D}$, and $\det h(z)$ has no zeros in $\overline{\D}$.}
\label{eq:C}
\end{equation}
The finite predictor 
coefficients $\phi_{n,j}\in\C^{d\times d}$, $j \in \{1,\dots,n\}$, 
of $\{X_k\}$ are defined by
\begin{equation}
P_{[-n,-1]}X_0=\phi_{n,1}X_{-1}+\cdots+\phi_{n,n}X_{-n},
\label{eq:forward-phi123}
\end{equation}
where, for $n\in\N$, $P_{[-n,-1]}X_{0}$ stands for the best 
linear predictor of the future 
value $X_{0}$ based on the finite past $\{X_{-n},\dots,X_{-1}\}$ 
(see Section \ref{sec:2} for the precise definition). 
The finite predictor coefficients $\phi_{n,j}$ are among the most basic quantities in the 
prediction theory for $\{X_k\}$.

The main aim of this paper is to derive a closed-form expression for the 
finite predictor coefficients $\phi_{n,j}$ of a 
multivariate ARMA process. More precisely, 
in the main result of this paper, i.e., Theorem \ref{thm:phirep116} below, we show that 
the finite predictor coefficients $\phi_{n,j}$ can be expressed in terms of 
several explicit matrices to be introduced in Section \ref{sec:4}, 
which are of fixed sizes independent of $n$, 
unlike, e.g., the matrices that appear in the Yule--Walker equations for $\phi_{n,j}$. 
See Example \ref{exmp:simple333} below that illustrates this point. 
The significance of the closed-form expression for $\phi_{n,j}$ 
is that it provides us with a linear-time algorithm to compute $\phi_{n,1},\dots,\phi_{n,n}$ 
(see Remark \ref{rem:O(n)333} below).

The closed-form expression for $\phi_{n,j}$ also provides us with a powerful tool to study problems 
concerning the asymptotic behavior of $\phi_{n,j}$. 
Among such problems, we show a result on the asymptotic behavior of the sum 
$\sum_{j=1}^n \Vert\phi_{n,j} - \phi_j\Vert$ as $n \to \infty$, where 
$\phi_j$ are the infinite predictor coefficients; see (\ref{eq:ipc574}) below. 
This sum appears, for example, in proving the consistency of the autoregressive
model fitting process and the corresponding autoregressive spectral density estimator (see Berk \cite{Ber}), 
and in proving the validity of autoregressive sieve bootstrap 
(see, e.g., B\"uhlmann \cite{Bu} and Kreiss et al.\ \cite{KPP}). 
Because of difficulties in finding the asymptotic behavior of 
$\sum_{j=1}^n \Vert\phi_{n,j} - \phi_j\Vert$ itself, 
Baxter's inequality 
\[
\sum_{j=1}^{n}\Vert \phi_{n,j} - \phi_j \Vert
\le K\sum_{j=n+1}^{\infty}\Vert \phi_j\Vert, \qquad K\in (0,\infty),
\]
in \cite{Bax} has been used instead. 
Under a mild condition on the multivariate ARMA process, 
the closed-form expression for $\phi_{n,j}$ now enables us 
to determine the precise asymptotic behavior 
of $\sum_{j=1}^n \Vert\phi_{n,j} - \phi_j\Vert$ as $n \to \infty$ (see Theorem \ref{thm:Bax323} below). 
It turns out that 
Baxter's inequality gives an asymptotically optimal 
bound of $\sum_{j=1}^n \Vert \phi_{n,j} - \phi_j \Vert$ in the sense that
\[
\lim_{n\to\infty} \frac{\sum_{j=1}^{n}\Vert \phi_{n,j} - \phi_j \Vert}{\sum_{j=n+1}^{\infty}\Vert \phi_j\Vert} 
\in (0,\infty)
\]
holds (see Corollary \ref{cor:appl6475} below).

The proof of the closed-form expression for $\phi_{n,j}$ is long. 
One important ingredient of the proof is the explicit representation of 
$\phi_{n,j}$ (see the proof of Theorem \ref{thm:phirep116} in \ref{sec:Append-D} below), 
which was obtained recently in 
Inoue et al.\ \cite{IKP2}, extending the earlier univariate result in Inoue and Kasahara \cite{IK06}; 
see also Inoue et al.\ \cite{IKP1} and Inoue and Kasahara \cite{IK18} for related work. 
To explain another important ingredient of the proof of 
the closed-form expression for $\phi_{n,j}$, we recall that, for $h:\T\to \C^{d\times d}$ 
satisfying (\ref{eq:farima529}) and (\ref{eq:C}), 
there exists $h_{\sharp}:\T\to \C^{d\times d}$ that satisfies (\ref{eq:C}) and
\begin{equation}
w(e^{i\theta})=h(e^{i\theta})h(e^{i\theta})^*=h_{\sharp}(e^{i\theta})^*h_{\sharp}(e^{i\theta}),
\qquad \theta\in [-\pi,\pi),
\label{eq:rational333}
\end{equation}
and that $h_{\sharp}$ is unique up to a constant unitary factor (see, e.g., \cite{IKP2}). 
We may take $h_{\sharp}=h$ for the univariate case $d=1$ but not so for $d\ge 2$. 
We show, in Theorem \ref{thm:hsharpinv123} below, that $h_{\sharp}^{-1}$ has the same poles 
with the same multiplicities as $h^{-1}$. This is a key finding in deriving 
the closed-form expression for $\phi_{n,j}$ when $d\ge 2$. 
We remark, however, that the closed-form expression for $\phi_{n,j}$ itself, 
i.e., Theorem \ref{thm:phirep116} below, 
is new even for univariate ($d=1$) ARMA processes.

We explain the difference between 
the explicit representation of $\phi_{n,j}$ in \cite{IKP2}, i.e., 
Theorem 5.4 in \cite{IKP2}, and the closed-form expression of $\phi_{n,j}$ in this paper. 
The representation in \cite{IKP2} holds both 
for long and short memory processes, and has several applications such as the proof of 
Baxter's inequality for multivariate long-memory processes in \cite{IKP2}. 
The representation of $\phi_{n,j}$ in \cite{IKP2} 
is, however, not a closed-form expression since it involves infinite series. 
In this paper, for multivariate ARMA processes, we transform the representation in \cite{IKP2} to 
a closed-form expression for $\phi_{n,j}$. 
The advantage of the latter is clear from the fact that it can be viewed as a linear-time algorithm to compute 
$\phi_{n,1},\dots,\phi_{n,n}$, as stated above.

This paper is organized as follows.
In Section \ref{sec:2}, we give preliminary definitions and basic facts.
In Section \ref{sec:3}, we prove the correspondence between the poles of $h^{-1}$ and $h_{\sharp}^{-1}$. 
In Section \ref{sec:4}, we introduce several matrices which are to become building blocks for 
the closed-form expression of $\phi_{n, j}$. 
In Section \ref{sec:5}, we present the main result, i.e., the closed-form expression for $\phi_{n, j}$. 
In Section \ref{sec:6}, we apply the closed-form expression for $\phi_{n, j}$ to 
derive the asymptotic behavior of $\sum_{j=1}^n \Vert\phi_{n,j} - \phi_j\Vert$ as $n \to \infty$. 
Finally, the Appendix contains the omitted proofs.


\section{Preliminaries}\label{sec:2}

Let $\D:=\{z\in\C :\vert z\vert<1\}$ denote the open unit disk in $\C$. 
Let $\C^{m\times n}$ be the set of all complex $m\times n$ matrices;
we write $\C^d$ for $\C^{d\times 1}$. 
We write $I_n$ for the $n\times n$ unit matrix.
For $a\in \C^{m\times n}$, $a^{\top}$ denotes the transpose of $a$, and
$\bar{a}$ and
$a^*$ the complex and Hermitian conjugates of $a$, respectively;
thus, in particular, $a^*:=\bar{a}^{\top}$. 
For $a\in\C^{d\times d}$, we write $\Vert a\Vert$ for the norm 
$\Vert a\Vert:=\sup_{u\in \C^d, \vert u\vert \le 1}\vert au\vert$, where 
$\vert u\vert:=(\sum_{i=1}^d\vert u^i\vert^2)^{1/2}$ denotes the Euclidean norm of
$u=(u^1,\dots,u^d)^{\top}\in \C^d$. 
We denote by $\ell_{2+}^{d\times d}$ the space of
$\C^{d\times d}$-valued sequences
$\{a_k\}_{k=0}^{\infty}$ such that $\sum_{k=0}^{\infty} \Vert a_k\Vert^2<\infty$. 
For $r\in [1,\infty)$, we write $L_r(\T)$ for the Lebesgue space of measurable functions $f:\T\to\C$
such that $\Vert f\Vert_r<\infty$, where
$\Vert f\Vert_r:=\{\int_{-\pi}^{\pi}\vert f(e^{i\theta})\vert^r d\theta/(2\pi)\}^{1/r}$.
Let $L_r^{m\times n}(\T)$ be the space of $\C^{m\times n}$-valued functions on
$\T$ whose entries belong to $L_r(\T)$.

For $d\in\N$,
let $\{X_k\}=\{X_k:k\in\Z\}$ be a $\C^d$-valued, centered,
weakly stationary process, defined on a probability space $(\Omega, \mcF, P)$, which
we shall simply call a $d$-variate stationary process. 
If there exists a positive $d\times d$ Hermitian matrix-valued function $w$ on $\T$, satisfying
$w\in L^{d\times d}_1(\T)$ and 
$E[X_m X_n^*] = \int_{-\pi}^{\pi}e^{-i(m-n)\theta}w(e^{i\theta})d\theta/(2\pi)$, 
$n,m\in\Z$, 
then we call $w$ the spectral density of $\{X_k\}$. 
Here and throughout this paper, we assume that 
$\{X_k\}$ is a $d$-variate ARMA process in the sense that $\{X_k\}$ satisfies 
the following condition:
\begin{equation}
\mbox{$\{X_k\}$ is a $d$-variate stationary process that has spectral density $w$ satisfying (\ref{eq:farima529}) with (\ref{eq:C}).}
\label{eq:ARMA}
\end{equation}

\begin{remark}
Suppose that $\{X_k\}$ is a $d$-variate, causal and invertible ARMA process in the sense 
of \cite{BD}, that is, a $\C^d$-valued, centered, weakly stationary process described by 
the ARMA equation
\[
\Phi(B) X_n = \Psi(B) Z_n,\qquad n\in\Z,
\]
where, for $r, s\in\N\cup\{0\}$ and $\Phi_i, \Psi_j\in\C^{d\times d}$, $i \in \{1,\dots,r\}$, $j \in \{1,\dots,s\}$, 
\[
\Phi(z) = I_d - z\Phi_1 - \cdots - z^r\Phi_r, 
\qquad 
\Psi(z) = I_d - z\Psi_1 - \cdots - z^s\Psi_s
\]
are $\C^{d\times d}$-valued polynomials satisfying 
$\det \Phi(z)\neq 0$ and $\det \Psi(z)\neq 0$ on $\overline{\D}$, 
$B$ is the backward shift operator defined by $B X_m=X_{m-1}$, 
and $\{Z_k : k\in\Z\}$ is a $d$-variate white noise, that is, a $d$-variate, centered process 
such that $E[Z_nZ_m^*]=\delta_{nm}\Sigma$ for some positive-definite $\Sigma\in\C^{d\times d}$. 
Then, $\{X_k\}$ is a $d$-variate ARMA process satisfying (\ref{eq:farima529}) with (\ref{eq:C}) for 
$h(z) = \Phi(z)^{-1} \Psi(z) \Sigma^{1/2}$. 
Conversely, we can show that any $d$-variate ARMA process $\{X_k\}$ satisfying (\ref{eq:farima529}) with (\ref{eq:C}) is described 
by the above type of ARMA equation.
\end{remark}

Write $X_k=(X^1_k,\dots,X^d_k)^{\top}$, and
let $V$ be the complex Hilbert space
spanned by all the entries $\{X^j_k: k\in\Z,\ j \in \{1,\dots,d\}\}$ in $L^2(\Omega, \mcF, P)$,
which has inner product $(x, y)_{V}:=E[x\overline{y}]$ and
norm $\Vert x\Vert_{V}:=(x,x)_{V}^{1/2}$.
For $J\subset \Z$ such as $\{n\}$,
$(-\infty,n]:=\{n,n-1,\dots\}$, $[n,\infty):=\{n,n+1,\dots\}$,
and $[m,n]:=\{m,\dots,n\}$ with $m\le n$, 
we write $V_J^X$ for the closed linear span of 
$\{X^j_k: j \in \{1,\dots,d\},\ k\in J\}$ in $V$. 
Let $(V_J^X)^{\bot}$ be the orthogonal complement of $V_J^X$ in $V$, and 
let $P_J$ and $P_J^{\perp}$ be the orthogonal projection operators of $V$ onto
$V_J^X$ and $(V_J^X)^{\perp}$, respectively.

Let $V^d$ be the space of $\C^d$-valued random variables on
$(\Omega, \mcF, P)$ whose entries belong to $V$.
The norm $\Vert x\Vert_{V^d}$ of $x=(x^1,\dots,x^d)^{\top}\in V^d$ is given by
$\Vert x\Vert_{V^d}:=(\sum_{i=1}^d \Vert x^i\Vert_V^2)^{1/2}$.
For $J\subset\mathbb{Z}$ and
$x=(x^1,\dots,x^d)^{\top}\in V^d$, we write $P_Jx$ for $(P_Jx^1, \dots, P_Jx^d)^{\top}$.
We define $P_J^{\perp}x$ in a similar way. 
For $n\in\mathbb{N}$ and $j \in \{1,\dots,n\}$,
the finite predictor coefficients $\phi_{n,j}\in \C^{d\times d}$ of $\{X_k\}$ are defined by 
(\ref{eq:forward-phi123}). 
For $x=(x^1,\dots,x^d)^{\top}$ and $y=(y^1,\dots,y^d)^{\top}$ in $V^d$, 
$\langle x,y\rangle:=E[xy^*]=
(
(x^i, y^j)_V
)_{1\le i, j\le d}
\in \C^{d\times d}$ 
stands for the Gram matrix of $x$ and $y$.

For $K\in\N$, let $p_1,\dots,p_K$ be distinct points in $\D\setminus\{0\}$. 
For $\mu \in \{1,\dots,K\}$ and $i\in\N$, we define $p_{{\mu}, i}: \N\cup\{0\} \to \C$ by
\begin{equation}
p_{{\mu},i}(k):=
\binom{k}{i-1} p_{\mu}^{k - i+1}, \qquad k\in\N\cup\{0\}.
\label{eq:p-func333}
\end{equation}
Notice that $p_{\mu,i}(0)=\binom{0}{i-1}p_{\mu}^{-i+1}=\delta_{i, 1}$. 
Take $m_{\mu}\in\N$ for $\mu \in \{1,\dots, K\}$ and let
\begin{equation}
M:=\sum_{\mu=1}^K m_{\mu}.
\label{eq:sum-mmu123}
\end{equation}

The next proposition will be used in Section \ref{sec:3} and \ref{sec:Append-B}.

\begin{proposition}\label{eq:plind123}
For $N\in\N\cup\{0\}$, 
the $M$ vectors $p_{{\mu},i}\in \C^{1\times M}$, ${\mu} \in \{1,\dots,K\}$, $i \in \{1,\dots,m_{{\mu}}\}$, defined by
\[
p_{{\mu},i}=(p_{{\mu},i}(N), p_{{\mu},i}(N+1), \dots, p_{{\mu},i}(N+M-1))
\]
are linearly independent.
\end{proposition}


\section{Correspondence between the poles of $h^{-1}$ and $h_{\sharp}^{-1}$}\label{sec:3}

In this section, we assume that $\{X_k\}$ satisfies (\ref{eq:ARMA}). 
Let $h$ and $h_{\sharp}$ be as in (\ref{eq:farima529}) and (\ref{eq:rational333}), 
respectively, both satisfying (\ref{eq:C}).

Since $h^{-1}$ also satisfies (\ref{eq:C}), we can write $h^{-1}(z)$ in the form 
\begin{equation}
h(z)^{-1} = - \rho_0 - \sum_{{\mu}=1}^{K} \sum_{j=1}^{m_{\mu}} \frac{1}{(1-\barp_{\mu}z)^j}\rho_{{\mu}, j} 
- \sum_{j=1}^{m_0} z^j \rho_{0,j},
\label{eq:hinverse162}
\end{equation}
where
\begin{equation}
\left\{
\begin{aligned}
&K\in\N\cup\{0\},\\
&p_{\mu}\in\D\setminus \{0\},\quad {\mu} \in \{1,\dots,K\},\qquad p_{\mu}\neq p_{\nu},\quad \mu\neq \nu,\\
&m_{\mu}\in\N,\quad {\mu} \in \{1,\dots,K\}, \qquad m_0\in\N\cup\{0\},\\
&\rho_{{\mu}, j}\in\C^{d\times d},\quad {\mu} \in \{0,\dots,K\},\ j \in \{1,\dots,m_{\mu}\}, 
\qquad \rho_{0} \in\C^{d\times d},\\
&\rho_{{\mu},m_{\mu}}\neq 0,\quad {\mu} \in \{0,\dots,K\}.
\end{aligned}
\right.
\label{eq:hinverse163}
\end{equation}
In fact, we can obtain the expression (\ref{eq:hinverse162}) from the partial fraction decompositions of 
the entries of $h(z)^{-1}$; see Example \ref{exmp:pfd123} below. 
We remark that the convention $\sum_{k=1}^0=0$ is adopted in the sums on the right-hand side of (\ref{eq:hinverse162}).

The next theorem shows that $h_{\sharp}^{-1}$ 
of a multivariate ARMA process has the same $m_0$ and the same poles 
with the same multiplicities as $h^{-1}$. 

\begin{theorem}\label{thm:hsharpinv123}
For $m_0$, $K$ and $(p_1, m_1), \dots, (p_K, m_K)$ in 
$(\ref{eq:hinverse162})$ with $(\ref{eq:hinverse163})$, 
$h_{\sharp}^{-1}$ has the form 
\begin{equation}
h_{\sharp}(z)^{-1} = - \rho_0^{\sharp} - 
\sum_{{\mu}=1}^{K} \sum_{j=1}^{m_{\mu}} \frac{1}{(1-\bar{p}_{\mu}z)^j}\rho_{{\mu}, j}^{\sharp} 
- \sum_{j=1}^{m_0} z^j \rho_{0,j}^{\sharp},
\label{eq:hsharpinv162}
\end{equation}
where
\begin{equation}
\left\{
\begin{aligned}
&\rho_{{\mu}, j}^{\sharp}\in\C^{d\times d}, \quad 
{\mu} \in \{0,\dots,K\},\ j \in \{1,\dots,m_{\mu}\}, \qquad 
\rho_{0}^{\sharp} \in\C^{d\times d},\\
&\rho_{{\mu},m_{\mu}}^{\sharp}\neq 0, \quad {\mu} \in \{0,\dots,K\}.
\end{aligned}
\right.
\label{eq:hsharpinv163}
\end{equation}
Moreover, we have
\begin{equation}
\rho_{\mu,m_{\mu}} h_{\sharp}(p_{\mu})^* = h(p_{\mu})^* \rho_{\mu,m_{\mu}}^{\sharp}, 
\qquad \mu \in \{0,\dots,K\}.
\label{eq:rho-h555}
\end{equation}
\end{theorem}

The first half of Theorem \ref{thm:hsharpinv123} is a key ingredient of 
the proof of Theorem \ref{thm:phirep116} below, 
while the relations (\ref{eq:rho-h555}) play an important role in the proof of Theorem \ref{thm:Bax323} below.

\begin{exmp}\label{exmp:pfd123}
For $p\in\D$, let
\[
h(z)=
\left(
\begin{matrix}
1 & 0\\
1/(1-\barp z) & 1
\end{matrix}
\right).
\]
Then $h$ satisfies (\ref{eq:C}). 
For this $h$, we can take
\[
h_{\sharp}(z)=r
\left(
\begin{matrix}
1 - \vert p\vert^2 & 1\\
-1 + \dfrac{1 - \vert p\vert^2}{1 - \bar pz} & -\vert p\vert^2 + \dfrac{1}{1-\barp z}
\end{matrix}
\right),
\]
where $r:=1/\sqrt{1 - \vert p\vert^2 + \vert p\vert^4}$ 
(see Example 3 in \cite{IKP2}). We have
\[
h(z)^{-1} 
= \left(
\begin{matrix}
1 & 0\\
-1/(1-\barp z) & 1
\end{matrix}
\right),
\quad
h_{\sharp}(z)^{-1}=r
\left(
\begin{matrix}
-\vert p\vert^2 + \dfrac{1}{1-\barp z} & -1\\
1 - \dfrac{1 - \vert p\vert^2}{1 - \bar pz} & 1 - \vert p\vert^2
\end{matrix}
\right),
\]
so that $K=1$, $m_0=0$, $m_1=1$, $p_1=p$, and
\[
\rho_0 =-
\left(
\begin{matrix}
1 & 0\\
0 & 1
\end{matrix}
\right),
\quad 
\rho_{1,1} =
\left(
\begin{matrix}
0  & 0\\
1 & 0
\end{matrix}
\right),\quad 
\rho^{\sharp}_0=
r
\left(
\begin{matrix}
\vert p\vert^2  & 1\\
-1  & -1 + \vert p\vert^2
\end{matrix}
\right),
\quad 
\rho^{\sharp}_{1,1}=
r
\left(
\begin{matrix}
-1 & 0\\
1 - \vert p\vert^2 & 0
\end{matrix}
\right).
\]
\end{exmp}


\section{Building block matrices}\label{sec:4}

In this section, we introduce and study some matrices that serve as building blocks for 
the closed-form expression of $\phi_{n, j}$. 
We assume that $\{X_k\}$ satisfies (\ref{eq:ARMA}). 
Let $h$ and $h_{\sharp}$ be as in (\ref{eq:farima529}) and (\ref{eq:rational333}), respectively, 
both satisfying (\ref{eq:C}). 
We also assume that $K\ge 1$ for $K$ in (\ref{eq:hinverse162}). 
This assumption implies that $\{X_k\}$ is a $d$-variate ARMA process 
that is not an AR process; see Remark \ref{rem:AR612} below. 
For $m_1,\dots,m_K$ in (\ref{eq:hinverse162}), we define $M$ by (\ref{eq:sum-mmu123}).

For $\mu \in \{1,\dots,K\}$, $i \in \{1,\dots,m_{\mu}\}$, and $n\in\N\cup\{0\}$, we define
\begin{equation}
\p_{{\mu}, i}(n):=p_{{\mu},i}(n)I_d\in \C^{d\times d}
\label{eq:p362}
\end{equation}
using $p_{{\mu},i}(n)$ in (\ref{eq:p-func333}). 
For $n\in \N\cup\{0\}$, we also define $\p_n \in \C^{dM\times d}$ by the 
following block representation:
\begin{equation}
\p_n
:=(\p_{1, 1}(n), \dots, \p_{1, m_1}(n)\ \vert\ \p_{2, 1}(n), \dots, \p_{2, m_2}(n)\ \vert\ 
\cdots\ \vert\ 
\p_{K, 1}(n), \dots, \p_{K, m_{K}}(n))^{\top}.
\label{eq:ppp123}
\end{equation}
Notice that
\begin{equation}
\p_0 = (I_d, 0, \dots, 0\ \vert\ I_d, 0, \dots, 0\ \vert\ 
\cdots\ \vert\ 
I_d, 0, \dots, 0)^{\top} \in \C^{dM\times d}.
\label{eq:p000}
\end{equation}
We define $\Lambda\in\C^{dM\times dM}$ by
\begin{equation}
\Lambda := \sum_{\ell=0}^{\infty} \p_{\ell} \p_{\ell}^*.
\label{eq:Cdef12}
\end{equation}
For ${\mu}, {\nu}\in\{1,2,\dots,K\}$, we define $\Lambda^{{\mu},{\nu}}\in \C^{dm_{\mu}\times dm_{\nu}}$ by the 
block representation
\[
\Lambda^{{\mu},{\nu}} := \left(
\begin{matrix}
\lambda^{{\mu}, {\nu}}(1, 1) & \lambda^{{\mu}, {\nu}}(1, 2) & \cdots & \lambda^{{\mu}, {\nu}}(1, m_{\nu}) \cr
\lambda^{{\mu}, {\nu}}(2, 1) & \lambda^{{\mu}, {\nu}}(2, 2) & \cdots & \lambda^{{\mu}, {\nu}}(2, m_{\nu}) \cr
\vdots                       &  \vdots                      & \ddots & \vdots                             \cr
\lambda^{\mu,\nu}(m_{\mu},1) & \lambda^{\mu,\nu}(m_{\mu},2) & \cdots & \lambda^{\mu,\nu}(m_{\mu},m_{\nu}) \cr
\end{matrix}
\right),
\]
where, for $i \in \{1,\dots,m_{\mu}\}$, $j \in \{1,\dots,m_{\nu}\}$,
\[
\lambda^{{\mu},{\nu}}(i, j) 
:= \sum_{r=0}^{j-1} 
\binom{i-1}{r} \binom{i+j-r-2}{i-1} 
\frac{p_{\mu}^{j - r -1}\overline{p}_{\nu}^{i-r-1}}{(1-p_{\mu}\overline{p}_{\nu})^{i+j-r-1}} I_d 
\in \C^{d\times d}.
\]

Here is a closed-form expression of $\Lambda$.

\begin{lemma}\label{lem:C625}
The matrix $\Lambda$ has the following block representation:
\[
\Lambda=\left(
\begin{matrix}
\Lambda^{1, 1}       &  \Lambda^{1, 2}   &  \cdots  & \Lambda^{1, K}  \cr
\Lambda^{2, 1}       &  \Lambda^{2, 2}   &  \cdots  & \Lambda^{2, K}  \cr
\vdots               &  \vdots           &  \ddots  & \vdots          \cr
\Lambda^{K, 1}       &  \Lambda^{K, 2}   &  \cdots  & \Lambda^{K, K}  \cr
\end{matrix}
\right).
\]
\end{lemma}

We define
\begin{equation}
\tilde{h}(z) := \{h_{\sharp}(\overline{z})\}^*.
\label{eq:outft628}
\end{equation}
Then $\tilde{h}$ satisfies (\ref{eq:C}). 
We define, respectively, the forward MA and AR coefficients $c_k$ and $a_k$
of $\{X_k\}$ by
\[
h(z)=\sum_{k=0}^{\infty}z^kc_k,\qquad
-h(z)^{-1}=\sum_{k=0}^{\infty}z^ka_k,\qquad z\in\D,
\]
and the backward MA and AR coefficients $\tilde{c}_k$
and $\tilde{a}_k$ of $\{X_k\}$ by
\[
\tilde{h}(z)=\sum_{k=0}^{\infty}z^k\tilde{c}_k,\qquad
-\tilde{h}(z)^{-1}=\sum_{k=0}^{\infty}z^k\tilde{a}_k,\qquad z\in\D.
\]
All of $\{c_k\}$, $\{a_k\}$, $\{\tilde{c}_k\}$ and $\{\tilde{a}_k\}$ are $\C^{d\times d}$-valued
sequences that decay exponentially fast to zero,
and we have 
$c_0a_0=\tilde{c}_0\tilde{a}_0=-I_d$. 
We have the AR representation 
$\sum_{k=-\infty}^na_{n-k}X_k + \varepsilon_n = 0$ 
and the infinite prediction formula 
$P_{(-\infty, -1]}X_0 = \sum_{k=1}^{\infty}\phi_kX_{-k}$, 
where
\begin{equation}
\phi_k:=c_0a_k \in \C^{d\times d},\qquad k\in\N.
\label{eq:ipc574}
\end{equation}
We call $\phi_k$ the infinite predictor coefficients of $\{X_k\}$.

\begin{remark}\label{rem:AR612}
If $K$ in (\ref{eq:hinverse163}) satisfies $K=0$, then $a_0=\rho_0$, 
$a_k=\rho_{0,k}\ (1\le k\le m_0)$ and $a_k=0\ (k\ge m_0+1)$. In particular, we have 
$\sum_{k=0}^{m_0}a_{k}X_{n-k} + \varepsilon_n = 0$ for $n\in\Z$. 
This implies that $P_{[-n,-1]}X_0=\phi_1X_{-1}+\cdots+\phi_{m_0}X_{-m_0}$ for 
$n\ge \max(m_0, 1)$ and $\phi_k$ in (\ref{eq:ipc574}). 
Therefore, the finite predictor coefficients $\phi_{n,j}$ in (\ref{eq:forward-phi123}) are 
trivially obtained. By this reason, we assume $K\ge 1$ in Sections \ref{sec:4}--\ref{sec:6}.
\end{remark}

For $\tilde{h}$ in (\ref{eq:outft628}), we see from 
Theorem \ref{thm:hsharpinv123} that
\begin{equation}
\tilde{h}(z)^{-1} = - \tilde{\rho}_{0} - 
\sum_{{\mu}=1}^{K} \sum_{j=1}^{m_{\mu}} \frac{1}{(1 - p_{\mu} z)^j} \tilde{\rho}_{{\mu}, j} 
- \sum_{j=1}^{m_0} z^j \tilde{\rho}_{0,j},
\label{eq:htildeinv162}
\end{equation}
where
\[
\tilde{\rho}_{0} := (\rho_{0}^{\sharp})^*, \qquad 
\tilde{\rho}_{\mu,j} := (\rho_{\mu, j}^{\sharp})^*, \quad 
\mu \in \{0,\dots,K\},\ j \in \{1,\dots,m_{\mu}\}.
\]

\begin{proposition}\label{prop:a516}
We have
\begin{align}
a_n         &= \sum_{{\mu}=1}^{K} \sum_{j=1}^{m_{\mu}} \binom{n+j-1}{j-1} \barp_{\mu}^n
\rho_{{\mu}, j}, \qquad n\ge m_0 + 1,
\label{eq:a363}\\
\tilde{a}_n &= \sum_{{\mu}=1}^{K} \sum_{j=1}^{m_{\mu}} \binom{n+j-1}{j-1} p_{\mu}^n
\tilde{\rho}_{{\mu}, j}, \qquad n\ge m_0 + 1.
\label{eq:atilde361}
\end{align}
Moreover, if $m_0\ge 1$, then we have
\begin{align}
a_n         &= \rho_{0,n} + \sum_{{\mu}=1}^{K} \sum_{j=1}^{m_{\mu}} \binom{n+j-1}{j-1} \barp_{\mu}^n
\rho_{{\mu}, j}, \qquad n \in \{1,\dots,m_0\},
\label{eq:a364}\\
\tilde{a}_n &= \tilde{\rho}_{0,n} + \sum_{{\mu}=1}^{K} \sum_{j=1}^{m_{\mu}} \binom{n+j-1}{j-1} p_{\mu}^n
\tilde{\rho}_{{\mu}, j},
 \qquad  n \in \{1,\dots,m_0\}.
\label{eq:atilde362}
\end{align}
\end{proposition}

\begin{proof}[\bf Proof]
Since
\begin{equation}
\frac{1}{(1 - q z)^j} = \sum_{n=0}^{\infty} \binom{n+j-1}{j-1} q^n z^n, \qquad 
q, z\in\D,\ j\in\N,
\label{eq:binom156}
\end{equation}
(\ref{eq:htildeinv162}) gives
\[
\tilde{h}(z)^{-1} = - \tilde{\rho}_{0} - 
\sum_{n=0}^{\infty} z^n \sum_{{\mu}=1}^{K} \sum_{j=1}^{m_{\mu}} \binom{n+j-1}{j-1} p_{\mu}^n
\tilde{\rho}_{{\mu}, j} 
- \sum_{j=1}^{m_0} z^j \tilde{\rho}_{0,j}.
\]
Thus, (\ref{eq:atilde361}) and (\ref{eq:atilde362}) follow. 
Similarly, we obtain (\ref{eq:a363}) and (\ref{eq:a364}) from (\ref{eq:hinverse162}) and (\ref{eq:binom156}).
\end{proof}

For $n\in\N$, we define $v_n, \tilde{v}_n \in\C^{dM\times d}$ by
\begin{align*}
v_n &:= \sum_{\ell = 0}^{\infty} \p_{\ell} a_{n + \ell},\\
\tilde{v}_n &:= \sum_{\ell = 0}^{\infty} \bar{\p}_{\ell} \tilde{a}_{n + \ell}.
\end{align*}
To give closed expressions for $v_n$ and $\tilde{v}_n$, we introduce some matrices. 
For $n\in\N$ and ${\mu}, {\nu}\in\{1,2,\dots,K\}$, 
we define $\Xi_n^{{\mu},{\nu}}\in \C^{dm_{\mu}\times dm_{\nu}}$ by 
the block representation
\[
\Xi_n^{{\mu},{\nu}} := \left(
\begin{matrix}
\xi_n^{{\mu}, {\nu}}(1, 1) & \xi_n^{{\mu}, {\nu}}(1, 2) & \cdots & \xi_n^{{\mu}, {\nu}}(1, m_{\nu}) \cr
\xi_n^{{\mu}, {\nu}}(2, 1) & \xi_n^{{\mu}, {\nu}}(2, 2) & \cdots & \xi_n^{{\mu}, {\nu}}(2, m_{\nu}) \cr
\vdots                     &  \vdots                    & \ddots & \vdots                           \cr
\xi_n^{\mu,\nu}(m_{\mu},1) & \xi_n^{\mu,\nu}(m_{\mu},2) & \cdots & \xi_n^{\mu,\nu}(m_{\mu},m_{\nu}) \cr
\end{matrix}
\right),
\]
where, for $n\in\N$, $i \in \{1,\dots,m_{\mu}\}$, $j \in \{1,\dots,m_{\nu}\}$, 
$\xi_n^{{\mu},{\nu}}(i, j) \in \C^{d\times d}$ is defined by
\[
\xi_n^{{\mu},{\nu}}(i, j)
:= \sum_{r=0}^{j-1} 
\binom{n+i+j-2}{r} \binom{i+j-r-2}{i-1} 
\frac{p_{\mu}^{j - r -1}\overline{p}_{\nu}^{n+i+j-r-2}}{(1-p_{\mu}\overline{p}_{\nu})^{i+j-r-1}} I_d.
\]
For $n\in\N$, we define $\Xi_n\in\C^{dM\times dM}$ by
\[
\Xi_n:=\left(
\begin{matrix}
\Xi_n^{1, 1}       &  \Xi_n^{1, 2}   &  \cdots  & \Xi_n^{1, K}  \cr
\Xi_n^{2, 1}       &  \Xi_n^{2, 2}   &  \cdots  & \Xi_n^{2, K}  \cr
\vdots             &  \vdots         &  \ddots  & \vdots        \cr
\Xi_n^{K, 1}       &  \Xi_n^{K, 2}   &  \cdots  & \Xi_n^{K, K}  \cr
\end{matrix}
\right).
\]
We also define $\rho \in \C^{dM\times d}$ and $\tilde{\rho} \in \C^{dM\times d}$ 
by the block representations
\[
\rho
:=(\rho_{1, 1}^{\top}, \dots, \rho_{1, m_1}^{\top} \ \vert \ 
\rho_{2, 1}^{\top}, \dots, \rho_{2, m_2}^{\top} \ \vert \ 
\cdots \ \vert \ 
\rho_{K, 1}^{\top}, \dots, \rho_{K, m_{K}}^{\top})^{\top}
\]
and
\[
\begin{aligned}
\tilde{\rho}
&:=(\tilde{\rho}_{1, 1}^{\top}, \dots, \tilde{\rho}_{1, m_1}^{\top} 
\ \vert \ \tilde{\rho}_{2, 1}^{\top}, \dots, \tilde{\rho}_{2, m_2}^{\top} \ \vert \ 
\cdots \ \vert \ 
\tilde{\rho}_{K, 1}^{\top}, \dots, \tilde{\rho}_{K, m_{K}}^{\top})^{\top}\\
&=\left(\overline{\rho_{1, 1}^{\sharp}}, \dots, \overline{\rho_{1, m_1}^{\sharp}} 
\ \vert \ \overline{\rho_{2, 1}^{\sharp}}, \dots, \overline{\rho_{2, m_2}^{\sharp}} \ \vert \ 
\cdots \ \vert \ 
\overline{\rho_{K, 1}^{\sharp}}, \dots, \overline{\rho_{K, m_{K}}^{\sharp}}\right)^{\top},
\end{aligned}
\]
respectively.

Here are closed-form expressions for $v_n$ and $\tilde{v}_n$.

\begin{lemma}\label{lem:v516}
We have
\begin{align}
v_n         &= \Xi_n \rho, \qquad n\ge m_0 + 1,
\label{eq:v222}\\
\tilde{v}_n &=  \overline{\Xi}_n \tilde{\rho}, \qquad n\ge m_0 + 1.
\label{eq:vtilde222}
\end{align}
Moreover, if $m_0\ge 1$, then we have
\begin{align}
v_n         &=  \Xi_n \rho + \sum_{\ell = 0}^{m_0-n} \p_{\ell} \rho_{0,n + \ell}, \qquad n \in \{1,\dots,m_0\},
\label{eq:v333}\\
\tilde{v}_n &=  \overline{\Xi}_n \tilde{\rho} + \sum_{\ell = 0}^{m_0-n} \overline{\p}_{\ell} \tilde{\rho}_{0,n + \ell},
 \qquad  n \in \{1,\dots,m_0\}.
\label{eq:vtilde333}
\end{align}
\end{lemma}

We define
\begin{equation}
h^{\dagger}(z) := h(1/\overline{z})^*
\label{eq:hdagger456}
\end{equation}
For $\mu \in \{0,1,\dots,K\}$, $j \in \{1,\dots,m_{\mu}\}$, we put
\begin{equation}
\theta_{{\mu}, j} := - \lim_{z\to p_{\mu}} \frac{1}{(m_{\mu} - j)!} 
\frac{d^{m_{\mu} - j}}{dz^{m_{\mu} - j}} \left\{(z-p_{\mu})^{m_{\mu}} h_{\sharp}(z) h^{\dagger}(z)^{-1}\right\} 
\in \C^{d\times d},
\label{eq:theta456}
\end{equation}
where $p_0:=0$. 
We define the block-diagonal matrix 
$\Theta\in \C^{dM\times dM}$ by
\begin{equation}
\Theta:=
\left(
\begin{matrix}
\Theta_{1}           &    0        &  \cdots &    0             \cr
0                    &  \Theta_{2} &  \cdots &    0             \cr
\vdots               &  \vdots     &  \ddots &   \vdots         \cr
0                    &    0        &  \cdots &  \Theta_{K}    
\end{matrix}
\right),
\label{eq:Theta555}
\end{equation}
where, for $\mu \in \{1,\dots,K\}$, $\Theta_{\mu} \in \C^{dm_{\mu}\times dm_{\mu}}$ is defined by
\begin{equation}
\Theta_{\mu} := \left(
\begin{matrix}
\theta_{{\mu}, 1} & \theta_{{\mu}, 2}  &  \cdots & \theta_{{\mu}, m_{\mu}-1} & \theta_{{\mu}, m_{\mu}} \cr
\theta_{{\mu}, 2} & \theta_{{\mu}, 3}  &  \cdots & \theta_{{\mu}, m_{\mu}}   &                         \cr
\vdots            &  \vdots            &         &                           &                         \cr
\theta_{{\mu}, m_{\mu}-1}   &  \theta_{{\mu}, m_{\mu}} &          &          &                         \cr
\theta_{{\mu}, m_{\mu}}     &                          &          &          & \mbox{\huge 0}
\end{matrix}
\right)
\label{eq:subTheta555}
\end{equation}
using $\theta_{\mu,j}$ in (\ref{eq:theta456}) with (\ref{eq:hdagger456}).

For $n\in \N\cup\{0\}$, we define the block-diagonal matrix 
$\Pi_n\in \C^{dM\times dM}$ by
\begin{equation}
\Pi_n:=
\left(
\begin{matrix}
\Pi_{1, n}           &    0        &  \cdots &    0             \cr
0                    &  \Pi_{2, n} &  \cdots &    0             \cr
\vdots               &  \vdots     &  \ddots &   \vdots         \cr
0                    &    0        &  \cdots &  \Pi_{K, n}    
\end{matrix}
\right),
\label{eq:Pin555}
\end{equation}
where, for $\mu \in \{1,\dots,K\}$ and $n\in\N\cup\{0\}$, $\Pi_{\mu, n} \in \C^{dm_{\mu}\times dm_{\mu}}$ is defined by
\begin{equation}
\Pi_{{\mu},n} := 
\left(
\begin{matrix}
\p_{{\mu}, 1}(n) & \p_{{\mu}, 2}(n) &  \p_{{\mu}, 3}(n)  &     \cdots   &    \p_{{\mu}, m_{\mu}}(n)   \cr
                 & \p_{{\mu}, 1}(n) &  \p_{{\mu}, 2}(n)  &     \cdots   &    \p_{{\mu}, m_{\mu}-1}(n) \cr
                 &                  &  \ddots            &     \ddots   &    \vdots                   \cr
                 &                  &                    &     \ddots   &    \p_{{\mu}, 2}(n)         \cr
\mbox{\huge 0}   &                  &                    &              &    \p_{{\mu}, 1}(n)
\end{matrix}
\right)
\label{eq:subPin555}
\end{equation}
using $\p_{{\mu}, i}(n)$ in (\ref{eq:p362}).


\section{Closed-form expression for finite predictor coefficients}\label{sec:5}

In this section, we assume that $\{X_k\}$, $h$ and $h_{\sharp}$ are as in Section \ref{sec:4}. 
Thus $\{X_k\}$ is a $d$-variate ARMA process satisfying (\ref{eq:ARMA}) and $K\ge 1$ for $K$ in 
(\ref{eq:hinverse162}). 
Recall the finite predictor coefficients 
$\phi_{n,k}\in \C^{d\times d}$ of the $d$-variate ARMA process 
$\{X_k\}$ from (\ref{eq:forward-phi123}). 
For $n\in \N\cup\{0\}$, we define $G_n, \tilde{G}_n\in \C^{dM\times dM}$ by
\begin{align}
G_n &:=  \Pi_n \Theta \Lambda, 
\label{eq:Gamma333}\\
\tilde{G}_n &:= (\Pi_n \Theta)^* \Lambda^{\top}.
\label{eq:tildeGamma333}
\end{align}

Here is the main theorem of this paper, which gives a closed-form expression for $\phi_{n,j}$.

\begin{theorem}\label{thm:phirep116}
For $n\ge \max(m_0, 1)$ and $j \in \{1,\dots,n\}$, we have
\begin{equation}
\phi_{n,j}
=c_0a_j + c_0 \p_0^{\top} (I_{dM} - \tilde{G}_n G_n)^{-1} (\Pi_n \Theta)^* 
\{\Lambda^{\top} \Pi_n \Theta v_j + \tilde{v}_{n-j+1}\}.
\label{eq:phi-rep3}
\end{equation}
\end{theorem}

Recall the assumption for Theorem \ref{thm:phirep116} from the beginning 
of this section; $\{X_k\}$ in Theorem \ref{thm:phirep116} is a general $d$-variate ARMA process 
that is not an AR process (see Remark \ref{rem:AR612} above). 
We remark that, from Lemma \ref{lem:Ginv123} below, 
$I_{dM} - \tilde{G}_n G_n$ is invertible for $n\ge m_0$.

\begin{corollary}\label{cor:phirep117}
If $m_0=0$, then, for $n\ge 1$ and $j \in \{1,\dots,n\}$, we have
\begin{equation}
\phi_{n,j}
=c_0a_j + c_0 \p_0^{\top} (I_{dM} - \tilde{G}_n G_n)^{-1} (\Pi_n \Theta)^* 
\{\Lambda^{\top} \Pi_n \Theta \Xi_j \rho + \overline{\Xi}_{n-j+1} \tilde{\rho}\}.
\label{eq:phi-rep331}
\end{equation}
\end{corollary}

\begin{proof}[\bf Proof]
The corollary follows immediately from Theorem \ref{thm:phirep116} and Lemma \ref{lem:v516}.
\end{proof}

The matrices $a_j$, $\p_0$, $\Pi_n$, and $\Theta$ in (\ref{eq:phi-rep3}) are 
given by the closed-form expressions (\ref{eq:a363}) and (\ref{eq:a364}), (\ref{eq:p000}), 
(\ref{eq:Pin555}) with (\ref{eq:subPin555}), and (\ref{eq:Theta555}) with (\ref{eq:subTheta555}), 
respectively. 
The closed-form expression of $\Lambda$, $v_n$ and $\tilde{v}_n$ 
are given by Lemmas \ref{lem:C625} and \ref{lem:v516}, and those of $G_n$ and $\tilde{G}_n$ by 
(\ref{eq:Gamma333}) and (\ref{eq:tildeGamma333}), respectively. 
Moreover, the matrix $c_0$ is given by 
$c_0=h(0)= - \{\rho_{0} + \sum_{{\mu}=1}^{K} \sum_{j=1}^{m_{\mu}} \rho_{{\mu}, j}\}^{-1}$. 
Therefore, (\ref{eq:phi-rep3}) gives a complete 
closed-form expression for $\phi_{n,j}$. 
Notice that the sizes of all the matrices are fixed and independent of $n$.

\begin{remark}\label{rem:phi333}
Notice that $c_0a_j=\phi_j$ in (\ref{eq:phi-rep3}) 
is the infinite predictor coefficient.
\end{remark}

\begin{exmp}\label{exmp:simple333}
Suppose that $m_{\mu}=1$, $\mu \in \{1,\dots,K\}$ and $m_0=0$, 
that is,
\[
h(z)^{-1} = - \rho_{0} - \sum_{{\mu}=1}^{K} \frac{1}{1-\barp_{\mu}z}\rho_{{\mu}, 1}, \qquad 
h_{\sharp}(z)^{-1} = - \rho^{\sharp}_{0} - 
\sum_{{\mu}=1}^{K} \frac{1}{1 - \barp_{\mu} z} \rho^{\sharp}_{{\mu}, 1}.
\]
Then, Corollary \ref{cor:phirep117} holds with 
$a_j = \sum_{{\mu}=1}^{K} \barp_{\mu}^j\rho_{{\mu}, 1}$ for $j\ge 1$, 
$\p_0^{\top} = (I_d, \dots, I_d)\in \C^{d \times dK}$,
\begin{align*}
\Theta &= \left(
\begin{matrix}
p_1 h_{\sharp}(p_1) \rho_{1, 1}^* &    0                               &  \cdots &    0                              \cr
0                                 &  p_2 h_{\sharp}(p_2) \rho_{2, 1}^* &  \cdots &    0                              \cr
\vdots                            &  \vdots                            &  \ddots &   \vdots                          \cr
0                                 &    0                               &  \cdots & p_K h_{\sharp}(p_K) \rho_{K, 1}^*
\end{matrix}
\right)
\in \C^{dK\times dK},\\
\Lambda &= \left(
\begin{matrix}
\frac{1}{1-p_{1}\overline{p}_{1}} I_d & \frac{1}{1-p_{1}\overline{p}_{2}} I_d 
&  \cdots  & \frac{1}{1-p_{1}\overline{p}_{K}} I_d  \cr
\frac{1}{1-p_{2}\overline{p}_{1}} I_d  &  \frac{1}{1-p_{2}\overline{p}_{2}} I_d  
&  \cdots  & \frac{1}{1-p_{2}\overline{p}_{K}} I_d  \cr
\vdots               &  \vdots           &  \ddots       & \vdots          \cr
\frac{1}{1-p_{K}\overline{p}_{1}} I_d & \frac{1}{1-p_{K}\overline{p}_{2}} I_d 
&  \cdots  & \frac{1}{1-p_{K}\overline{p}_{K}} I_d  \cr
\end{matrix}
\right)
\in \C^{dK\times dK},\\
\Pi_n &= \left(
\begin{matrix}
p_1^n I_d  &    0        &  \cdots &    0             \cr
0          & p_2^n I_d   &  \cdots &    0             \cr
\vdots     &  \vdots     &  \ddots &   \vdots         \cr
0          &    0        &  \cdots &   p_K^n I_d  
\end{matrix}
\right)
\in \C^{dK\times dK},\qquad n\ge 0,\\
\Xi_n &= \left(
\begin{matrix}
\frac{\overline{p}_{1}^{n}}{1-p_{1}\overline{p}_{1}} I_d & \frac{\overline{p}_{2}^{n}}{1-p_{1}\overline{p}_{2}} I_d 
&  \cdots  & \frac{\overline{p}_{K}^{n}}{1-p_{1}\overline{p}_{K}} I_d  \cr
\frac{\overline{p}_{1}^{n}}{1-p_{2}\overline{p}_{1}} I_d  &  \frac{\overline{p}_{2}^{n}}{1-p_{2}\overline{p}_{2}} I_d  
&  \cdots  & \frac{\overline{p}_{K}^{n}}{1-p_{2}\overline{p}_{K}} I_d  \cr
\vdots               &  \vdots           &  \ddots       & \vdots          \cr
\frac{\overline{p}_{1}^{n}}{1-p_{K}\overline{p}_{1}} I_d & \frac{\overline{p}_{2}^{n}}{1-p_{K}\overline{p}_{2}} I_d 
&  \cdots  & \frac{\overline{p}_{K}^{n}}{1-p_{K}\overline{p}_{K}} I_d  \cr
\end{matrix}
\right)
\in \C^{dK\times dK}, \qquad n\ge 1,\\
\rho
&=(\rho_{1, 1}^{\top}, \rho_{2, 1}^{\top}, \dots, \rho_{K, 1}^{\top})^{\top} \in \C^{dK\times d}, 
\qquad 
\tilde{\rho}
=\left(\overline{\rho^{\sharp}_{1, 1}}, \overline{\rho^{\sharp}_{2, 1}}, 
\dots, \overline{\rho^{\sharp}_{K, 1}}\right)^{\top}
\in \C^{dK\times d}
\end{align*}
and $G_n =  \Pi_n \Theta \Lambda$, $\tilde{G}_n = (\Pi_n \Theta)^* \Lambda^{\top} \in \C^{dK\times dK}$.
\end{exmp}

\begin{remark}\label{rem:O(n)333}
We define the block-diagonal matrix 
$J\in \C^{dM\times dM}$ by
\[
J:=
\left(
\begin{matrix}
J_{1}                &    0       &  \cdots &    0             \cr
0                    &  J_{2}     &  \cdots &    0             \cr
\vdots               &  \vdots    &  \ddots &   \vdots         \cr
0                    &    0       &  \cdots &  J_{K}    
\end{matrix}
\right),
\]
where, for $\nu \in \{1,\dots,K\}$, $J_{\nu} \in \C^{dm_{\nu}\times dm_{\nu}}$ is defined by
\[
J_{\nu} := 
\left(
\begin{matrix}
\overline{p}_{\nu}I_d & I_d                   &          &            &  \mbox{\huge 0}           \cr
                      & \overline{p}_{\nu}I_d &  I_d     &            &                           \cr
                      &                       &  \ddots  &   \ddots   &                           \cr
                      &                       &          &   \ddots   &    I_d                    \cr
\mbox{\huge 0}        &                       &          &            &  \overline{p}_{\nu} I_d
\end{matrix}
\right), \quad m_{\nu}\ge 2, 
\qquad 
:= \overline{p}_{\nu}I_d, \quad m_{\nu}=1.
\]
Then it is easy to see that 
$\Xi_{n+1} = \Xi_n J$ for $n\in\N$. 
By this recursion, we can compute $\Xi_1,\dots,\Xi_n$ in $O(n)$ arithmetic operations. 
The other matrices in (\ref{eq:phi-rep3}) and (\ref{eq:phi-rep331}) can also be computed in $O(n)$ operations. 
Therefore, we see that the complexity of the algorithm to compute $\phi_{n,1},\dots,\phi_{n,n}$ 
that is provided by Theorem \ref{thm:phirep116} 
or Corollary \ref{cor:phirep117} is only $O(n)$, which is the best possible. 
Notice that $(\phi_{n,n},\phi_{n,n-1}\dots,\phi_{n,1})$ is the solution to the Yule--Walker equation
\[
(\phi_{n,n},\phi_{n,n-1},\dots,\phi_{n,1}) T_n(w)  = (\gamma(-n), \gamma(-n+1),\dots,\gamma(-1))
\]
or
\[
T_n(w) (\phi_{n,n},\phi_{n,n-1},\dots,\phi_{n,1})^* = (\gamma(-n), \gamma(-n+1),\dots,\gamma(-1))^*,
\]
where $T_n(w)$ is 
the truncated block Toeplitz matrix defined by
\[
T_n(w)
:=\left(
\begin{matrix}
\gamma(0)   & \gamma(-1)  & \cdots & \gamma(-n+1)\cr
\gamma(1)   & \gamma(0)   & \cdots & \gamma(-n+2)\cr
\vdots      & \vdots      & \ddots & \vdots      \cr
\gamma(n-1) & \gamma(n-2) & \cdots & \gamma(0)
\end{matrix}
\right)
\in \C^{dn\times dn}.
\]
Also notice that the multivariate Durbin--Levinson recursion solves the Yule-Walker equation 
in $O(n^2)$ time (see, e.g., Brockwell and Davis \cite{BD}). 
Algorithms for Toeplitz linear systems that run faster than $O(n^2)$ 
are called superfast; see Xi et al.\ \cite{XXCB} and the references therein.
\end{remark}

\begin{remark}
From the discussions in Remark \ref{rem:O(n)333}, we are naturally led to the problem of finding 
linear-time algorithms to compute the solution $x\in \C^{dn\times d}$ of the general block Toeplitz system 
$T_n(w) x=b$ for $b\in\C^{dn\times d}$ 
and $w$ satisfying (\ref{eq:farima529}) with (\ref{eq:C}). 
This problem will be solved in \cite{I}.
\end{remark}

\begin{remark}
One possible application of Theorem \ref{thm:phirep116} is model fitting. 
More precisely, suppose that we are given a dataset $x_1,\dots,x_N$ as a realization of the underlying process $\{X_k\}$. 
Then, for suitable $n$, we search for the parameters of the ARMA model that minimize the least squares
error 
$\sum_{m=n+1}^{N} \vert x_m - \sum_{k=1}^n \phi_{n,k} x_{m-k} \vert^2$, using Theorem \ref{thm:phirep116}. 
In this way, we simultaneously fit the ARMA model to the data and estimate the predictor 
coefficients $\phi_{n,1},\dots,\phi_{n,n}$, without estimating the autocovariance function $\gamma$. 
The validity of this method will be discussed in future work.
\end{remark}


\section{Application}\label{sec:6}

We continue to assume that 
$\{X_k\}$ is a $d$-variate ARMA process satisfying (\ref{eq:ARMA}) and $K\ge 1$ for $K$ in 
(\ref{eq:hinverse162}). 
In this section, we further assume
\begin{equation}
\vert p_1\vert > \max \{\vert p_{\mu}\vert: \mu \in \{2,\dots,K\}\},
\label{eq:p1max}
\end{equation}
and apply Theorem \ref{thm:phirep116} above to 
determine the asymptotic behavior of $\sum_{j=1}^n \Vert \phi_{n,j} - \phi_j \Vert$ 
as $n\to\infty$. 
We write $s_n \sim t_n$ as $n\to\infty$ to mean that $\lim_{n\to\infty} s_n/t_n = 1$. 

\begin{theorem}\label{thm:Bax323}
We assume $(\ref{eq:p1max})$. 
Then 
\begin{equation}
\sum_{j=1}^n \Vert \phi_{n,j} - \phi_j \Vert 
\ \sim \ 
\frac{C_1}{(m_1 - 1)!} n^{m_1 - 1} \vert p_1 \vert^{n}\quad \mbox{as $n\to\infty$},
\label{eq:Baxeq111}
\end{equation}
where $C_1$ is a positive constant given by 
$C_1 := 
\sum_{k=1}^{\infty} \Vert c_0 h(p_1)^* \rho_{1, m_1}^{\sharp} H \tilde{v}_{k} \Vert$ 
with 
$H := \left( I_d, 0, \dots, 0 \right) \in \C^{d\times dM}$.
\end{theorem}

\begin{proof}[\bf Proof]
First we show that the constant $C_1$ is in $(0, \infty)$. 
We define 
$C_{1,k} := \Vert c_0 h(p_1)^* \rho_{1, m_1}^{\sharp} H \tilde{v}_{k} \Vert$ for $k \in \N$, 
so that 
$C_1 := \sum_{k=1}^{\infty} C_{1,k}$ 
holds. 
Then, the sum converges since $C_{1,k}$ decays exponentially fast 
as $k\to\infty$. Therefore, it is enough to show that $C_{1,k}>0$ for $k$ large enough. 
By Lemma \ref{lem:v516}, we have, for $k\ge m_0 + 1$,
\[
\begin{aligned}
C_{1,k}
&= \Vert c_0 h(p_1)^* \rho_{1, m_1}^{\sharp} H \overline{\Xi}_k \tilde{\rho} \Vert
= \left\Vert 
c_0 h(p_1)^* \rho_{1, m_1}^{\sharp} 
\left( \sum\nolimits_{\nu = 1}^K \sum\nolimits_{j = 1}^{m_{\nu}} \overline{\xi}_k^{1,{\nu}}(1, j) (\rho_{\nu,j}^{\sharp})^* 
\right)
\right\Vert \\
&= \left\Vert 
\sum\nolimits_{\nu = 1}^K \sum\nolimits_{j = 1}^{m_{\nu}} \overline{\xi}_k^{1,{\nu}}(1, j)
c_0 h(p_1)^* \rho_{1, m_1}^{\sharp} (\rho_{\nu,j}^{\sharp})^* 
\right\Vert 
\ge 
k^{m_1 - 1} \vert p_1 \vert^k (A_k - B_k),
\end{aligned}
\]
where
\begin{align*}
A_k &:= \left\Vert 
(k^{m_1 - 1} \vert p_1 \vert^{k})^{-1}
\overline{\xi}_k^{1, 1}(1, m_1)
c_0 h(p_1)^* \rho_{1, m_1}^{\sharp} (\rho_{1, m_1}^{\sharp})^* 
\right\Vert,\\
B_k &:= \left\Vert 
\sum\nolimits_{(\nu, j) \neq (1, m_1)} 
(k^{m_1 - 1} \vert p_1 \vert^{k})^{-1}
\overline{\xi}_k^{1,{\nu}}(1, j)
c_0 h(p_1)^* \rho_{1, m_1}^{\sharp} (\rho_{\nu,j}^{\sharp})^* 
\right\Vert.
\end{align*}
The main term in 
\[
\overline{\xi}_k^{1,1}(1, m_1)
=\sum_{r=0}^{m_1-1} 
\binom{k+m_1-1}{r} 
\frac{\overline{p}_1^{m_1 - r} p_1^{k+m_1-1-r}}{(1-\vert p_1\vert^2)^{m_1-r}} I_d
\]
is $\binom{k+m_1-1}{m_1-1} p_1^{k}(1-\vert p_1\vert^2)^{-1} I_d$ for $r = m_1 - 1$ 
and we have 
\[
\lim_{k \to \infty} (k^{m_1 - 1} 
\vert p_1 \vert^{k})^{-1}\overline{\xi}_k^{1, 1}(1, m_1) 
= 
\{(m_1 - 1)! (1-\vert p_1\vert^2)\}^{-1} I_d,
\]
so that $\lim_{k \to \infty} A_k = A_{\infty}$, where 
$A_{\infty} 
:= 
\{(m_1 - 1)! (1-\vert p_1\vert^2)\}^{-1} 
\Vert 
c_0 h(p_1)^* \rho_{1, m_1}^{\sharp} (\rho_{1, m_1}^{\sharp})^* 
\Vert$. 
Since $c_0 h(p_1)^*$ is invertible and $\rho_{1, m_1}^{\sharp} (\rho_{1, m_1}^{\sharp})^* \neq 0$, 
we have $A_{\infty} >0$. 
On the other hand, (\ref{eq:p1max}) implies 
\[
\lim_{k \to \infty} 
(k^{m_1 - 1} \vert p_1 \vert^{k})^{-1}
\overline{\xi}_k^{1,{\nu}}(1, j) 
= 0
\]
for $(\nu, j) \neq (1, m_1)$. Hence $\lim_{k \to \infty} B_k = 0$. 
Combining, we see that $C_{1,k}>0$ for $k$ large enough, as desired.

Next we prove (\ref{eq:Baxeq111}). Recall $p_{{\mu},i}(n)$ and $\p_{{\mu}, i}(n)$ from 
(\ref{eq:p-func333}) and (\ref{eq:p362}), respectively. 
Since (\ref{eq:p1max}) implies
\[
\lim_{n\to\infty} \frac{1}{p_{1,m_1}(n)} \p_{{\mu}, i}(n)
=
\begin{cases}
I_d, & \mu=1, \  i=m_1, \\
0,   & \mbox{otherwise},
\end{cases}
\]
we have 
$\lim_{n\to\infty} (1/p_{1,m_1}(n)) \Pi_n = \Delta$, 
where $\Delta\in \C^{dM\times dM}$ is defined by
\[
\Delta
:=\left(
\begin{matrix}
\Delta_1  &    0     &  \cdots &    0      \cr
0         &    0     &  \cdots &    0      \cr
\vdots    &  \vdots  &  \ddots &   \vdots  \cr
0         &    0     &  \cdots &    0
\end{matrix}
\right),
\qquad 
\Delta_1
:=\left(
\begin{matrix}
0      &   \cdots    &  0       &   I_d     \cr
0      &   \cdots    &  0       &    0      \cr
\vdots &             &  \vdots  &   \vdots  \cr
0      &    \cdots   &  0       &    0
\end{matrix}
\right)
\in \C^{dm_1\times dm_1}.
\]
Hence, by Theorem \ref{thm:phirep116} and the dominated convergence theorem, we get
\[
\begin{aligned}
&\frac{1}{\vert p_{1,m_1}(n)\vert} \sum_{j=1}^n \Vert \phi_{n,j} - \phi_j \Vert
=\frac{1}{\vert p_{1,m_1}(n)\vert} \sum_{j=1}^n 
\Vert  
c_0 \p_0^{\top} (I_{dM} - \tilde{G}_n G_n)^{-1} (\Pi_n \Theta)^* 
\{\Lambda^{\top} \Pi_n \Theta v_j + \tilde{v}_{n-j+1}\}
\Vert                                                                             \\
&\qquad = \sum_{k=1}^n 
\left\Vert  
c_0 \p_0^{\top} (I_{dM} - \tilde{G}_n G_n)^{-1} \left(\frac{1}{p_{1,m_1}(n)}\Pi_n \Theta\right)^* 
\{\Lambda^{\top} \Pi_n \Theta v_{n+1-k} + \tilde{v}_{k}\}
\right\Vert\\
&\qquad \to \quad 
\sum_{k=1}^{\infty} 
\Vert  
c_0 \p_0^{\top}  (\Delta \Theta)^* \tilde{v}_{k}
\Vert, \quad 
n \to \infty.
\end{aligned}
\]
By simple calculations, we have
\[
\Delta \Theta
=\left(
\begin{matrix}
(p_1)^{m_1} h_{\sharp}(p_1) \rho_{1, m_1}^*  &    0     &  \cdots &    0      \cr
0                                            &    0     &  \cdots &    0      \cr
\vdots                                       &  \vdots  &  \ddots &   \vdots  \cr
0                                            &    0     &  \cdots &    0
\end{matrix}
\right)
\in \C^{dM\times dM},
\]
so that 
$\p_0^{\top}  (\Delta \Theta)^*
= (p_1)^{m_1} \rho_{1, m_1} h_{\sharp}(p_1)^* H$. 
However, (\ref{eq:rho-h555}) implies that 
$\rho_{1,m_1} h_{\sharp}(p_1)^* = h(p_1)^* \rho_{1,m_1}^{\sharp}$. Hence, we see that 
$\sum_{k=1}^{\infty} 
\Vert  
c_0 \p_0^{\top}  (\Delta \Theta)^* \tilde{v}_{k}
\Vert
= C_1$. 
Thus (\ref{eq:Baxeq111}) follows. 
\end{proof}

\begin{corollary}\label{cor:appl6475}
We assume $(\ref{eq:p1max})$. 
Then 
\begin{equation}
\lim_{n \to \infty} 
\frac{\sum_{j=1}^n \Vert \phi_{n,j} - \phi_j \Vert}
{\sum_{k=n+1}^{\infty}\Vert \phi_k\Vert}
= 
\frac{(1 - \vert p_1 \vert) C_1}{\vert p_1 \vert \cdot \Vert c_0 \rho_{1, m_1}\Vert}.
\label{eq:Baxeq222}
\end{equation}
\end{corollary}

\begin{proof}[\bf Proof]
By (\ref{eq:ipc574}), Proposition \ref{prop:a516} and (\ref{eq:p1max}), we have
\[
\Vert \phi_k\Vert
= \Vert c_0 a_k\Vert
= 
\left\Vert 
\sum\nolimits_{{\mu}=1}^{K} \sum\nolimits_{j=1}^{m_{\mu}} \binom{k+j-1}{j-1} \barp_{\mu}^k 
c_0 \rho_{{\mu}, j} 
\right\Vert
\sim \ \Vert c_0 \rho_{1, m_1}\Vert \binom{k+m_1-1}{m_1-1} \vert p_1 \vert^k, 
\qquad k\to\infty.
\]
Hence, 
$\sum_{k=n+1}^{\infty}\Vert \phi_k\Vert
\ \sim \ \Vert c_0 \rho_{1, m_1}\Vert \sum_{k=n+1}^{\infty} \binom{k+m_1-1}{m_1-1} \vert p_1 \vert^k$ 
as $n\to\infty$. 
From
\[
\sum_{k=n+1}^{\infty} \binom{k+m_1-1}{m_1-1} x^k 
= \frac{1}{(m_1 - 1)!} \left(\frac{d}{dx}\right)^{m_1-1}\left( \frac{x^{n+m_1}}{1-x} \right), 
\qquad \vert x\vert <1,
\]
and Leibniz's rule, we have, as $k\to\infty$,
\[
\sum_{k=n+1}^{\infty} \binom{k+m_1-1}{m_1-1} \vert p_1 \vert^k 
\ \sim \ \binom{n+m_1}{m_1 - 1} \frac{\vert p_1 \vert^{n+1}}{1-\vert p_1 \vert}
\ \sim \ \frac{n^{m_1-1} \vert p_1 \vert^{n+1}}{(m_1 - 1)! (1-\vert p_1 \vert)}.
\]
Thus 
\begin{equation}
\sum_{k=n+1}^{\infty}\Vert \phi_k\Vert
\ \sim \ 
\frac{\Vert c_0 \rho_{1, m_1}\Vert}{(m_1 - 1)! (1-\vert p_1 \vert)} n^{m_1-1} \vert p_1 \vert^{n+1}, 
\qquad n\to\infty.
\label{eq:phi-sum-asymp111}
\end{equation} 
The assertion 
(\ref{eq:Baxeq222}) follows from (\ref{eq:phi-sum-asymp111}) and Theorem \ref{thm:Bax323}.
\end{proof}

\begin{remark}
To explain the assumption (\ref{eq:p1max}), 
we consider two parameters $p_1=x_1+iy_1$ and 
$p_2=x_2+iy_2$ belonging to the space $A:=\{(p_1,p_2)\in (\D\setminus \{0\})^2: \vert p_1\vert \ge \vert p_2\vert\}$. 
Then, the arrangement $\vert p_1\vert > \vert p_2\vert$ is generic in the sense that 
the complement
\[
\{(p_1,p_2)\in A: \vert p_1\vert = \vert p_2\vert\}=\{(p_1,p_2)\in A: x_1^2+y_1^2=x_2^2+y_2^2\}
\]
forms a hypersurface, hence its 4-dimensional Lebesgue measure is zero. 
In the same sense, the arrangement of $(p_1,\dots,p_K)$ given by (\ref{eq:p1max}) is generic. For, 
without loss of generality, we may assume
\[
\vert p_1\vert\ge \max \{\vert p_{\mu}\vert: \mu = 2,\dots,K\}.
\]
Then, 
if $(p_1,\dots,p_K)$ does not satisfy (\ref{eq:p1max}), then we have 
$\vert p_1\vert = \vert p_\mu\vert$ for some $\mu \in \{2,\dots,K\}$. 
Here, it should be noticed that the special choice of $p_1$ in (\ref{eq:p1max}) is just for the sake of simplicity; 
an analogue of Theorem \ref{thm:Bax323}, hence Corollary \ref{cor:appl6475}, still holds 
even if we replace (\ref{eq:p1max}) by, e.g., 
\[
\vert p_K\vert > \max \{\vert p_{\mu}\vert: \mu \in \{1,\dots,K-1\}\}.
\]
Still, it will be interesting to pursue analogues of Theorem \ref{thm:Bax323} and Corollary \ref{cor:appl6475} when 
(\ref{eq:p1max}) fails to hold, hence oscillations of the type 
$k_1 p_1^n + k_2(e^{i\theta}p_1)^n$ occur as $n\to\infty$.
\end{remark}



\appendix

\section{Proof of Proposition \ref{eq:plind123}}\label{sec:Append-A}

For $f:\N\cup\{0\}\to \C$ and ${\mu}\in\{1,2,\dots,K\}$, we define 
$D_{\mu}f:\N\cup\{0\}\to \C$ by
\[
D_{\mu}f(k) := f(k+1) - p_{\mu}f (k),\qquad k\in\N\cup\{0\}.
\]

\begin{proposition}\label{prop:diff312}
For ${\mu}, {\nu}\in\{1,2,\dots,K\}$, $i\in\N$ and $k\in\N\cup\{0\}$, 
\begin{equation}
D_{\nu} p_{{\mu},i}(k) =
(p_{\mu} - p_{\nu}) p_{{\mu},i}(k) + p_{{\mu},i-1}(k),
\label{eq:diff514}
\end{equation}
where $p_{{\mu}, 0}\equiv 0$. 
\end{proposition}

\begin{proof}[\bf Proof]
Since
\[
D_{\nu} p_{{\mu},1}(k) = p_{\mu}^{k+1} - p_{\nu}p_{\mu}^k=(p_{\mu} - p_{\nu})p_{{\mu},1}(k),
\]
(\ref{eq:diff514}) holds for $i=1$. 
If $i\ge 2$, then, 
Pascal's rule $\binom{k+1}{i-1} = \binom{k}{i-1} + \binom{k}{i-2}$ implies that
\[
D_{\nu} p_{{\mu},i}(k)
=\binom{k+1}{i-1}p_{\mu}^{k-i+2} - \binom{k}{i-1} p_{\nu}p_{\mu}^{k-i+1}
=(p_{\mu} - p_{\nu}) p_{{\mu},i}(k) + p_{{\mu},i-1}(k).
\]
Thus (\ref{eq:diff514}) follows. 
\end{proof}

\begin{proof}[\bf Proof of Proposition \ref{eq:plind123}]
Let $\gamma_{{\mu},i}\in\C$, ${\mu} \in \{1,\dots,K\}$, $i \in \{1,\dots,m_{\mu}\}$ and 
suppose that 
\[
\sum_{{\mu}=1}^{K}\sum_{i=1}^{m_{\mu}} \gamma_{{\mu},i}p_{{\mu},i}(k)=0,\qquad k \in \{N,\dots,N+M-1\}.
\]
By Proposition \ref{prop:diff312}, we have
\[
0=\left( D_1^{m_1-1}D_2^{m_2}\cdots D_{K}^{m_{K}}
\sum_{{\mu}=1}^{K}\sum_{i=1}^{m_{\mu}} \gamma_{{\mu},i}p_{{\mu},i} \right)(N)
=\gamma_{1,m_1}p_1^N\prod_{{\mu}=2}^{K}(p_1-p_{\mu})^{m_{\mu}}.
\]
Hence $\gamma_{1,m_1}=0$. Repeating this procedure, we find that $\gamma_{{\mu},i}=0$, 
${\mu} \in \{1,\dots,K\}$, $i \in \{1,\dots,m_{\mu}\}$. Thus $p_{{\mu},i}$'s are linearly independent.
\end{proof}


\section{Proof of Theorem \ref{thm:hsharpinv123}}\label{sec:Append-B}

As in Section \ref{sec:3}, we assume that $\{X_k\}$ satisfies (\ref{eq:ARMA}). 
Let $h$ and $h_{\sharp}$ be as in (\ref{eq:farima529}) and (\ref{eq:rational333}), 
respectively, both satisfying (\ref{eq:C}). 

We consider the unitary matrix valued function $h^*h_{\sharp}^{-1}=h^{-1}h_{\sharp}^*$ on $\T$, 
called the phase function of $\{X_k\}$ (see p.\ 428 in Peller \cite{Pe}). 
We define a sequence $\{\beta_k\}_{k=-\infty}^{\infty}$ 
as the (minus of the) Fourier coefficients of $h^*h_{\sharp}^{-1}=h^{-1}h_{\sharp}^*$:
\begin{equation}
\beta_k 
=
-\int_{-\pi}^{\pi}e^{-ik\theta} h(e^{i\theta})^* h_{\sharp}(e^{i\theta})^{-1} \frac{d\theta}{2\pi}
=
-\int_{-\pi}^{\pi}e^{-ik\theta} h(e^{i\theta})^{-1} h_{\sharp}(e^{i\theta})^* \frac{d\theta}{2\pi}, 
\qquad k\in\Z.
\label{eq:beta-def667}
\end{equation}
From (\ref{eq:beta-def667}), we have
\begin{equation}
\beta_k^* 
=
-\int_{-\pi}^{\pi}e^{ik\theta} 
\{h_{\sharp}(e^{i\theta})^*\}^{-1} h(e^{i\theta}) \frac{d\theta}{2\pi}
=
-\int_{-\pi}^{\pi}e^{ik\theta} 
h_{\sharp}(e^{i\theta}) \{h(e^{i\theta})^*\}^{-1} \frac{d\theta}{2\pi}, 
\qquad k\in\Z.
\label{eq:beta*667}
\end{equation}
The proof of Theorem \ref{thm:hsharpinv123} below is based on the calculations of $\beta_k$ 
in two different ways.

Recall $h^{\dagger}$ from (\ref{eq:hdagger456}). 
From (\ref{eq:hinverse162}), we have
\begin{equation}
h^{\dagger}(z)^{-1}
= - \rho_0^* - \sum_{{\mu}=1}^{K} \sum_{j=1}^{m_{\mu}} \frac{z^j}{(z-p_{\mu})^j}\rho_{{\mu}, j}^* 
- \sum_{j=1}^{m_0} z^{-j} \rho_{0, j}^*.
\label{eq:hdagger123}
\end{equation}
Since $h(e^{i\theta})^* =h^{\dagger}(e^{i\theta}) $, 
we see from (\ref{eq:beta*667}) that
\begin{equation}
\beta_k^* 
= -\int_{-\pi}^{\pi} 
e^{ik\theta} h_{\sharp}(e^{i\theta})h^{\dagger}(e^{i\theta})^{-1}  \frac{d\theta}{2\pi},
\qquad k\in\Z.
\label{eq:beta621}
\end{equation}
Notice that the entries of $h_{\sharp}(z) h^{\dagger}(z)^{-1}$ 
are rational functions of $z\in\C$. 

Recall $\theta_{{\mu}, j}$ from (\ref{eq:theta456}).

\begin{proposition}\label{prop:hh316}
The matrix function $h_{\sharp}(z) h^{\dagger}(z)^{-1}$ has the form
\[
h_{\sharp}(z) h^{\dagger}(z)^{-1}
=
- \sum_{{\mu}=1}^{K} \sum_{j=1}^{m_{\mu}} \frac{1}{(z-p_{\mu})^j} \theta_{{\mu}, j} 
- \sum_{j=1}^{m_0} z^{-j} \theta_{0, j}
- R(z),
\]
where $R(z)$ is a $d\times d$ matrix function whose entries are rational functions of $z$ 
with no poles in $\overline{\D}$. Moreover, we have 
\begin{equation}
\theta_{\mu, m_{\mu}} = 
\begin{cases} (p_{\mu})^{m_{\mu}}h_{\sharp}(p_{\mu})\rho_{\mu, m_{\mu}}^* \neq 0, & 
\mu \in \{1,\dots,K\}, \\
h_{\sharp}(0)\rho_{0, m_0}^* \neq 0, & \mu=0.
\end{cases}
\label{eq:theta515}
\end{equation}
\end{proposition}

\begin{proof}[\bf Proof]
From (\ref{eq:hdagger123}), we have
\[
\begin{aligned}
- h_{\sharp}(z) h^{\dagger}(z)^{-1} 
&= h_{\sharp}(z) \rho_0^* + 
\sum_{{\mu}=1}^{K} \sum_{j=1}^{m_{\mu}} \frac{1}{(z-p_{\mu})^j} z^j h_{\sharp}(z) \rho_{{\mu}, j}^* 
+ \sum_{j=1}^{m_0} z^{-j} h_{\sharp}(z) \rho_{0,j}^*\\
&=  \sum_{{\mu}=1}^{K} \sum_{j=1}^{m_{\mu}} \frac{1}{(z-p_{\mu})^j} \theta_{{\mu}, j} 
+ \sum_{j=1}^{m_0} z^{-j} \theta_{0, j}
+ R(z),
\end{aligned}
\]
where $R(z)$ is 
a $d\times d$ matrix valued function 
whose entries are rational functions of $z$ with no poles in $\overline{\D}$. In particular, we have 
$\theta_{0, m_0} = h_{\sharp}(0)\rho_{0, m_0}^*$ and 
$\theta_{\mu, m_{\mu}} = (p_{\mu})^{m_{\mu}}h_{\sharp}(p_{\mu})\rho_{\mu, m_{\mu}}^*$, 
$\mu \in \{1,\dots,K\}$. 
Since $\rho_{0,m_0}\ne 0$ and $h_{\sharp}(0)$ 
is invertible, we see that $\theta_{0, m_0} \ne 0$. Similarly, $\theta_{\mu, m_{\mu}} \ne 0$, 
$\mu \in \{1,\dots,K\}$.
\end{proof}

\begin{proposition}\label{prop:beta654}
We have 
$\beta_{n+1}^* 
= \sum_{{\mu}=1}^{K} \sum_{j=1}^{m_{\mu}} \binom{n}{j - 1} p_{\mu}^{n - j +1} \theta_{{\mu}, j} 
+ \sum_{j=1}^{m_0} \delta_{n+1,j} \theta_{0, j}$ for $n\in\N\cup\{0\}$. 
In particular, 
$\beta_{n+1}^* 
= \sum_{{\mu}=1}^{K} \sum_{j=1}^{m_{\mu}} \binom{n}{j - 1} p_{\mu}^{n - j +1} \theta_{{\mu}, j}$ 
for $n\ge m_0$.
\end{proposition}

\begin{proof}[\bf Proof]
By (\ref{eq:beta621}), Proposition \ref{prop:hh316} and Cauchy's formula, we have, for $n\in \N\cup\{0\}$,
\[
\begin{aligned}
\beta_{n+1}^*
&= - \int_{\T} \zeta^n 
h_{\sharp}(\zeta) h^{\dagger}(\zeta)^{-1} \frac{d\zeta}{2\pi i}
= \sum_{{\mu}=1}^{K} \sum_{j=1}^{m_{\mu}}  
\int_{\T}\frac{\zeta^n}{(\zeta - p_{\mu})^j} \frac{d\zeta}{2\pi i}  \theta_{{\mu}, j} 
+ \sum_{j=1}^{m_0}  \int_{\T} \zeta^{n-j} \frac{d\zeta}{2\pi i}  \theta_{0, j} 
+ \int_{\T} \zeta^n R(\zeta) \frac{d\zeta}{2\pi i}\\
&=\sum_{{\mu}=1}^{K} \sum_{j=1}^{m_{\mu}} \binom{n}{j - 1} p_{\mu}^{n - j +1} \theta_{{\mu}, j}
+ \sum_{j=1}^{m_0} \delta_{n+1,j} \theta_{0, j}.
\end{aligned}
\]
Thus, the proposition follows.
\end{proof}

\begin{proof}[\bf Proof of Theorem \ref{thm:hsharpinv123}]
As in (\ref{eq:hinverse162}) with (\ref{eq:hinverse163}), 
we can write $h_{\sharp}(z)^{-1}$ in the form
\[
h_{\sharp}(z)^{-1} 
= - \sigma_0 - 
\sum_{{\mu}=1}^{L} \sum_{j=1}^{n_{\mu}} 
\frac{1}{(1-\overline{r}_{\mu} z)^j}\sigma_{{\mu}, j} 
- \sum_{j=1}^{n_0} z^j \sigma_{0,j},
\]
where
\[
\left\{
\begin{aligned}
&L\in\N\cup\{0\},\\
&r_{\mu}\in\D\setminus \{0\},\quad {\mu} \in \{1, \dots, L\},
\qquad r_{\mu} \neq r_{\nu},\quad \mu\neq \nu,\\
&n_{\mu}\in\N,\quad {\mu} \in \{1, \dots, L\}, \qquad 
n_0 \in \N\cup\{0\},\\
&\sigma_{{\mu}, j} \in \C^{d\times d}, 
\quad {\mu} \in \{0, \dots, L\},\ j \in \{1, \dots, n_{\mu}\}, 
\qquad \sigma_{0} \in\C^{d\times d},\\
&\sigma_{{\mu}, n_{\mu}}\neq 0,\quad {\mu} \in \{0,\dots,L\}.
\end{aligned}
\right.
\]
We put 
$r_{0}:=0$ and $h_{\sharp}^{\dagger}(z):=\{h_{\sharp}(1/\overline{z})\}^*$. 
We follow the argument in the proof of Proposition \ref{prop:beta654} above by using 
$\beta_k^* =
- \int_{-\pi}^{\pi}e^{ik\theta} 
\{h_{\sharp}(e^{i\theta})^*\}^{-1} h(e^{i\theta}) d\theta/(2\pi)$
instead of 
$\beta_k^* =
- \int_{-\pi}^{\pi}e^{ik\theta} 
h_{\sharp}(e^{i\theta}) \{h(e^{i\theta})^*\}^{-1} d\theta/(2\pi)$ 
to calculate $\beta^*_{n+1}$. Then, 
\begin{equation}
\beta_{n+1}^* = \sum_{{\mu}=1}^{L} \sum_{j=1}^{n_{\mu}} 
\binom{n}{j - 1} r_{\mu}^{n - j +1} \lambda_{{\mu}, j} 
+ \sum_{j=1}^{n_0} \delta_{n+1,j} \lambda_{0, j},\qquad n\in\N\cup\{0\},
\label{eq:beta246}
\end{equation}
where 
\[
\lambda_{{\mu}, j} = - \lim_{z\to r_{\mu}} \frac{1}{(n_{\mu} - j)!} 
\frac{d^{n_{\mu} - j}}{dz^{n_{\mu} - j}} \left\{(z-r_{\mu})^{n_{\mu}} 
h_{\sharp}^{\dagger}(z)^{-1}  h(z) \right\}
\in \C^{d\times d}, \qquad \mu \in \{0,\dots,L\}, \ j \in \{1,\dots,n_{\mu}\}.
\]
We also obtain
\begin{equation}
\lambda_{\mu, n_{\mu}} = 
\begin{cases}
(r_{\mu})^{n_{\mu}} \sigma_{\mu, n_{\mu}}^* h(r_{\mu}) \neq 0, & 
\mu \in \{1,\dots,L\}, \\
\sigma_{0, n_0}^*h(0) \neq 0, & \mu=0.
\end{cases}
\label{eq:lambda515}
\end{equation}
From Proposition \ref{prop:beta654} and (\ref{eq:beta246}), we have
\[
\sum_{{\mu}=1}^{K} \sum_{j=1}^{m_{\mu}} \binom{n}{j - 1} p_{\mu}^{n - j +1} \theta_{{\mu}, j} 
+ \sum_{j=1}^{m_0} \delta_{n+1,j} \theta_{0, j}
=\sum_{{\mu}=1}^{L} \sum_{j=1}^{n_{\mu}} 
\binom{n}{j - 1} r_{\mu}^{n - j +1} \lambda_{{\mu}, j} 
+ \sum_{j=1}^{n_0} \delta_{n+1,j} \lambda_{0, j}, \qquad n\in\N\cup\{0\}.
\]
In particular, 
$\sum_{{\mu}=1}^{K} \sum_{j=1}^{m_{\mu}} \binom{n}{j - 1} p_{\mu}^{n - j +1} \theta_{{\mu}, j} 
=
\sum_{{\mu}=1}^{L} \sum_{j=1}^{n_{\mu}} 
\binom{n}{j - 1} r_{\mu}^{n - j +1} \lambda_{{\mu}, j}$ 
for $n\ge \max(m_0, n_0)$. This and Proposition \ref{eq:plind123} yield 
$K=L$, $p_{\mu} = r_{f(\mu)}$, $m_{\mu} = n_{f(\mu)}$ and 
$\theta_{\mu, j} = \lambda_{f(\mu), j}$ for 
$\mu \in \{1,\dots,K\}$, $j \in \{1,\dots,m_{\mu}\}$ and 
some bijection $f: \{1,\dots,K\}\to \{1,\dots,K\}$. We now have 
$\sum_{j=1}^{m_0} \delta_{n+1,j} \theta_{0, j}
=\sum_{j=1}^{n_0} \delta_{n+1,j} \lambda_{0, j}$ for $n\in\N\cup\{0\}$, 
and this gives $m_0=n_0$ (as well as $\theta_{0, j} = \lambda_{0, j}$, 
$j \in \{1,\dots, m_0\}$). Thus, (\ref{eq:hsharpinv162}) and (\ref{eq:hsharpinv163}) hold with 
$\rho_0^{\sharp}=\sigma_0$ and $\rho_{\mu,j}^{\sharp} = \sigma_{f(\mu),j}$, $\mu \in \{0,\dots,K\}$, 
$j \in \{1,\dots,m_{\mu}\}$. 
Finally, we obtain (\ref{eq:rho-h555}) from $\theta_{\mu, m_{\mu}} = \lambda_{f(\mu), m_{\mu}}$, 
(\ref{eq:theta515}) and (\ref{eq:lambda515}).
\end{proof}


\section{Proofs of Lemmas \ref{lem:C625} and \ref{lem:v516}}\label{sec:Append-C}

To prove Lemma \ref{lem:C625}, we use the next proposition.

\begin{proposition}\label{prop:sum123}
For $i, j, n\in \N\cup\{0\}$ and $x, y\in \D$, we have
\[
\sum_{\ell=0}^{\infty} \binom{\ell}{i} \binom{\ell + n}{j} x^{\ell-i} y^{\ell+n-j}
=\sum_{r=0}^{j} \binom{n+i}{r} 
\binom{i+j-r}{i} \frac{x^{j - r}y^{n+i-r}}{(1-xy)^{i+j+1-r}}.
\]
\end{proposition}

\begin{proof}[\bf Proof]
Let $i, j, n\in \N\cup\{0\}$ and $x, y\in \D$. 
Since $y^n/(1-xy)=\sum_{\ell=0}^{\infty} x^{\ell}y^{n+\ell}$, we have
\[
\frac{1}{i! j!} \left(\frac{\partial }{\partial y}\right)^{j}\left(\frac{\partial }{\partial x}\right)^{i} 
\frac{y^n}{1-xy}
=\sum_{\ell=0}^{\infty} \binom{\ell}{i} \binom{n+\ell}{j} x^{\ell-i} y^{n+\ell-j}.
\]
On the other hand, since 
$(1/r!) (d/d y)^{r}y^{n+i}= \binom{n+i}{r} y^{n+i-r}$ 
and
\[
\frac{1}{(j-r)!} \left(\frac{\partial }{\partial y}\right)^{j-r}\frac{1}{(1-xy)^{i+1}}
= \binom{i+j-r}{j-r} \frac{x^{j-r}}{(1-xy)^{i+j+1-r}}
= \binom{i+j-r}{i} \frac{x^{j-r}}{(1-xy)^{i+j+1-r}},\qquad j\ge r,
\]
we have
\[
\begin{aligned}
&\frac{1}{i! j!} 
\left(\frac{\partial }{\partial y}\right)^{j}\left(\frac{\partial }{\partial x}\right)^{i} 
\frac{y^n}{1-xy}
=
\frac{1}{j!} \left(\frac{\partial }{\partial y}\right)^{j}
\frac{y^{n+i}}{(1-xy)^{i+1}}\\
&\qquad =
\sum_{r=0}^{j} \binom{j}{r}\binom{j}{r}^{-1} \left\{ \frac{1}{r!} 
\left(\frac{\partial }{\partial y}\right)^{r}y^{n+i} \right\} \left\{
\frac{1}{(j-r)!}  \left(\frac{\partial }{\partial y}\right)^{j-r}\frac{1}{(1-xy)^{i+1}}\right\}\\
&\qquad =
\sum_{r=0}^{j} 
\binom{n+i}{r} \binom{i+j-r}{i} \frac{x^{j - r}y^{n+i-r}}{(1-xy)^{i+j+1-r}}.
\end{aligned}
\]
Comparing, we obtain the proposition.
\end{proof}

\begin{remark}
Notice that Proposition \ref{prop:sum123} with $n=0$ implies
\[
\sum_{r=0}^{j} \binom{i}{r} 
\binom{i+j-r}{i} \frac{x^{j - r}y^{i-r}}{(1-xy)^{i+j+1-r}}
 =\sum_{r=0}^{i} \binom{j}{r} 
\binom{i+j-r}{j} \frac{x^{j - r}y^{i-r}}{(1-xy)^{i+j+1-r}}.
\]
Also, notice that 
$\binom{i}{r} \binom{i+j-r}{i} =\binom{j}{r} \binom{i+j-r}{j}$.
\end{remark}

\begin{proof}[\bf Proof of Lemma \ref{lem:C625}]
The proof is immediate from (\ref{eq:Cdef12}) and Proposition \ref{prop:sum123} 
with $n=0$, and $i$ and $j$ replaced by $i-1$ and $j-1$, respectively.
\end{proof}

\begin{proof}[\bf Proof of Lemma \ref{lem:v516}]
If $n\ge m_0+1$, then Proposition \ref{prop:sum123} yields, 
for $\mu \in \{1,\dots,K\}$, $i \in \{1,\dots,m_{\mu}\}$, 
\[
\sum_{\ell=0}^{\infty} \p_{\mu,i}(\ell) a_{\ell+n}
= \sum_{{\nu}=1}^{K} \sum_{j=1}^{m_{\nu}} 
\left\{
\sum_{\ell=0}^{\infty}
\binom{\ell}{i-1} \binom{n+\ell+j-1}{j-1}  p_{\mu}^{\ell-i+1} \overline{p}_{\nu}^{n+\ell} 
\right\}
\rho_{\nu,j}
= \sum_{{\nu}=1}^{K} \sum_{j=1}^{m_{\nu}} \xi_n^{{\mu},{\nu}}(i, j) \rho_{\nu,j}
\]
and
\[
\sum_{\ell=0}^{\infty} \overline{\p}_{\mu,i}(\ell) \tilde{a}_{\ell+n}
= \sum_{{\nu}=1}^{K} \sum_{j=1}^{m_{\nu}} 
\left\{
\sum_{\ell=0}^{\infty}
\binom{\ell}{i-1} \binom{n+\ell+j-1}{j-1}  \overline{p}_{\mu}^{\ell-i+1} p_{\nu}^{n+\ell} 
\right\}
\tilde{\rho}_{\nu,j}
= \sum_{{\nu}=1}^{K} \sum_{j=1}^{m_{\nu}} \overline{\xi}_n^{{\mu},{\nu}}(i, j) \tilde{\rho}_{\nu,j}.
\]
Thus, (\ref{eq:v222}) and (\ref{eq:vtilde222}) follow. 
If $m_0\ge 1$ and $1\le n\le m_0$, then, similarly, we have (\ref{eq:v333}) and (\ref{eq:vtilde333}). 
\end{proof}


\section{Proof of Theorem \ref{thm:phirep116}}\label{sec:Append-D}

To prove Theorem \ref{thm:phirep116}, we first prepare some propositions and lemmas. 
Recall $\p_n$ from (\ref{eq:ppp123}).

\begin{proposition}\label{prop:linde567}
For $N\in\N\cup\{0\}$, the matrix 
$(\p_N,\p_{N+1},\dots,\p_{N+M-1})\in\C^{dM\times dM}$ 
is invertible.
\end{proposition}

\begin{proof}[\bf Proof]
For $k\in\N\cup\{0\}$, we define $p(k)\in\C^M$ by 
\[
p(k) = 
(p_{1, 1}(k), \dots, p_{1, m_1}(k) \vert p_{2, 1}(k), \dots, p_{2, m_2}(k) \vert 
 \cdots \vert 
p_{K, 1}(k), \dots, p_{K, m_{K}}(k))^{\top}.
\]
Then, by the definition of determinant, we have
\[
\det (\p_N,\p_{N+1},\dots,\p_{N+M-1})
=\left\{\det(p(N), p(N+1),\dots,p(N+M-1))\right\}^d.
\]
Since Proposition \ref{eq:plind123} implies that $\det(p(N), p(N+1),\dots,p(N+M-1))\neq 0$, 
the assertion follows.
\end{proof}

The next proposition will be used in the proof of Lemma \ref{lem:Ginv123} below.

\begin{proposition}\label{prop:Lambda116}
The matrix $\Lambda$ is positive definite. In particular, $\Lambda$ is invertible.
\end{proposition}

\begin{proof}[\bf Proof]
Clearly, $\Lambda$ is a Hermitian matrix. 
Suppose that $v\Lambda v^*=0$ for $v\in\C^{1\times dM}$. 
Since $v\p_{\ell} \p_{\ell}^* v^* = v\p_{\ell}(v\p_{\ell})^* \ge 0$, 
we see that $v\p_{\ell}=0$ for any $\ell\in\N\cup\{0\}$. This implies $v(\p_0,\p_1,\dots,\p_{M-1})=0$. 
Since $(\p_0,\p_1,\dots,\p_{M-1})\in \C^{dM\times dM}$ is invertible by Proposition \ref{prop:linde567}, 
we have $v=0$. Thus, $\Lambda$ is positive definite. 
\end{proof}

Let 
$X_k=\int_{-\pi}^{\pi}e^{-ik\theta}\eta(d\theta)$, $k\in\Z$, 
be the spectral representation of $\{X_k\}$, where $\eta$ is a $\C^d$-valued
random spectral measure. 
We define a $d$-variate stationary process $\{\varepsilon_k:k\in\Z\}$, 
called the forward innovation
process of $\{X_k\}$, by
\[
\varepsilon_k:=\int_{-\pi}^{\pi}e^{-ik\theta}h(e^{i\theta})^{-1}\eta(d\theta),\qquad k\in\Z.
\]
Then, $\{\varepsilon_k\}$ satisfies 
$\langle \varepsilon_n, \varepsilon_m\rangle = \delta_{n m}I_d$ and 
$V_{(-\infty,n]}^X=V_{(-\infty,n]}^{\varepsilon}$ for $n\in\Z$,
hence
\begin{equation}
(V_{(-\infty,n]}^X)^{\bot} = V_{[n+1, \infty)}^{\varepsilon},\qquad n\in\Z.
\label{eq:epsilon555}
\end{equation}
We also define the backward innovation process $\{\tilde{\varepsilon}_k: k\in\Z\}$ of $\{X_k\}$ by
\[
\tilde{\varepsilon}_k:=\int_{-\pi}^{\pi}e^{ik\theta}\{h_{\sharp}(e^{i\theta})^*\}^{-1}
\eta(d\theta),\qquad k\in\Z.
\]
Then, $\{\tilde{\varepsilon}_k\}$ satisfies
$\langle \tilde{\varepsilon}_n, \tilde{\varepsilon}_m\rangle=\delta_{n m}I_d$ and 
$V_{[-n,\infty)}^X=V_{(-\infty, n]}^{\tilde{\varepsilon}}$ for $n\in\Z$, 
hence
\begin{equation}
(V_{[-n,\infty)}^X)^{\bot} = V_{[n+1,\infty)}^{\tilde{\varepsilon}},\qquad n\in\Z.
\label{eq:tildeepsilon555}
\end{equation}

For $n\in\N\cup\{0\}$, we define 
$\mathcal{H}_n: (V_{[-n,\infty)}^X)^{\bot} \to (V_{(-\infty,-1]}^X)^{\bot}$ by
\[
\mathcal{H}_n x := P_{(-\infty,-1]}^{\perp} x,\qquad x\in (V_{[-n,\infty)}^X)^{\bot},
\]
and
$\tilde{\mathcal{H}}_n: (V_{(-\infty,-1]}^X)^{\bot} \to (V_{[-n,\infty)}^X)^{\bot}$ 
by
\[
\tilde{\mathcal{H}}_n x := P_{[-n,\infty)}^{\perp} x,\qquad x \in (V_{(-\infty,-1]}^X)^{\bot}.
\]
We denote by $\Vert \mathcal{H}_n\Vert$ (resp., $\Vert \tilde{\mathcal{H}}_n\Vert$) the operator norm of 
$\mathcal{H}_n$ (resp., $\tilde{\mathcal{H}}_n$).

\begin{proposition}\label{prop:Hbdd123}
For $n\in\N\cup\{0\}$, we have $\Vert \mathcal{H}_n\Vert=\Vert \tilde{\mathcal{H}}_n\Vert<1$.
\end{proposition}

\begin{proof}[\bf Proof]
Let $\{X'_k: k\in\Z\}$ be the dual process of $\{X_k\}$, which is 
a $d$-variate stationary process characterized by the biorthogonality relation 
$\langle X_j,X'_k\rangle=\delta_{jk}I_d$; see Masani \cite{M60} and Section 5 in 
\cite{IKP2}. 
The process $\{X'_k\}$ admits the two MA representations 
$X'_n=-\sum_{k=0}^\infty a_k^* \varepsilon_{n+k}$ and 
$X'_{-n}=-\sum_{k=0}^\infty \tilde{a}_k^* \tilde{\varepsilon}_{n+k}$ for $n\in\Z$. 
Moreover, for the spectral density $w$ of $\{X_k\}$, $\{X'_k\}$ has the spectral density $w^{-1}$. 
For $n\ge 0$, let 
\[
\rho_n:=\sup \{\vert (x, y)_V\vert: x\in V_{(-\infty,-n-1]}^{X'},\ y\in V_{[0,\infty)}^{X'},\ 
\Vert x\Vert_V\le 1,\ \Vert y\Vert_V\le 1\}
\]
be the cosine of angle between $V_{(-\infty,-n-1]}^{X'}$ and $V_{[0,\infty)}^{X'}$ 
(see, e.g., Treil and Volberg \cite{TV1, TV2}, Pourahmadi \cite{Pou}, and Bingham \cite{Bin}). 
Since both 
$w$ and $w^{-1}$ are continuous, hence bounded, on $\T$, $w^{-1}$ satisfies 
the matrix Muckenhoupt condition
\[
\sup_I 
\left\Vert 
\left( \frac{1}{m(I)} \int_I w^{-1} dm \right)^{1/2} 
\left( \frac{1}{m(I)} \int_I w dm \right)^{1/2}
\right\Vert<\infty,
\]
where $m$ is the normalized ($m(\T)=1$) Lebesgue measure on $\T$ and 
the supremum is taken over all subarcs $I$ of $\T$. 
Therefore, by Treil and Volberg \cite{TV1} (see also Peller \cite{Pe}, Arov and Dym \cite{AD}, 
and Bingham \cite{Bin}), 
we have $\rho_n<1$ for $n\ge 0$. 
Since both $-\sum_{k=0}^{\infty}z^ka_k^* = \{h(\overline{z})^*\}^{-1}$ and 
$-\sum_{k=0}^{\infty}z^k\tilde{a}_k^* = h_{\sharp}(z)^{-1}$ are outer 
(see, e.g., Katsnelson and Kirstein \cite{KK} and Section 2 in \cite{IKP2}), 
we see from (\ref{eq:epsilon555}) and (\ref{eq:tildeepsilon555}) that 
$V_{[0,\infty)}^{X'}=V_{[0,\infty)}^{\varepsilon}=(V_{(-\infty,-1]}^X)^{\bot}$ and that 
$V_{(-\infty,-n-1]}^{X'}=V_{[n+1,\infty)}^{\tilde{\varepsilon}}=(V_{[-n,\infty)}^X)^{\bot}$. 
Therefore,
\[
\rho_n 
= \sup \{\vert (x, y)_V\vert: x\in(V_{[-n,\infty)}^X)^{\bot},\ y\in (V_{(-\infty,-1]}^X)^{\bot}, 
\Vert x\Vert_V\le 1,\ \Vert y\Vert_V\le 1\}
=\Vert \mathcal{H}_n\Vert = \Vert \tilde{\mathcal{H}}_n\Vert
\]
(see Remark \ref{rem:angle122} below for the second and third equalities), so that 
$\Vert \mathcal{H}_n\Vert = \Vert \tilde{\mathcal{H}}_n\Vert<1$ for $n\ge 0$, as desired.
\end{proof}

\begin{remark}\label{rem:angle122}
For two closed subspaces $A$ and $B$ of a Hilbert space $L$, 
let $P_A: L\to A$ be the orthogonal projection operator and 
$P_A\vert_{B}$ the restriction of $P_A$ to $B$. Then we have 
$\sup \{\vert (x, y)\vert: x\in A,\ y\in B,\ \Vert x\Vert\le 1,\ \Vert y\Vert\le 1\}
= \Vert P_A\vert_{B}\Vert$. 
\end{remark}

The next lemma plays a key role in the arguments below.

\begin{lemma}\label{lem:beta162}
For $n\ge m_0$ and $k, \ell\in\N\cup\{0\}$, we have 
$\beta_{n+k+\ell+1}^* = \p_{\ell}^{\top} \Pi_n \Theta \p_k$, 
hence 
$\beta_{n+k+\ell+1} = \p_k^* (\Pi_n \Theta)^* \overline{\p}_{\ell}$.
\end{lemma}

\begin{proof}[\bf Proof]
We have
\[
\begin{aligned}
\sum_{j=1}^{\infty} \binom{n+k+\ell}{j-1} x^{j-1}
&=(1+x)^{n+k+\ell}=(1+x)^k(1+x)^{\ell}(1+x)^n\\
&=\sum_{j=1}^{\infty} 
\left\{
\sum_{r=0}^{j-1}\binom{k}{j-1-r} \sum_{s=0}^r \binom{\ell}{s}\binom{n}{r-s}
\right\} x^{j-1}\\
&=\sum_{j=1}^{\infty} 
\left\{
\sum_{i=1}^{j}\binom{k}{j-i} \sum_{q=1}^i \binom{\ell}{q-1}\binom{n}{i-q}
\right\} x^{j-1},
\end{aligned}
\]
where we have used the substitutions $i=r+1$ and $q=s+1$. 
Hence 
$\binom{n+k+\ell}{j-1} = \sum_{i=1}^{j}\binom{k}{j-i} \sum_{q=1}^i \binom{\ell}{q-1}\binom{n}{i-q}$ 
for $j\in\N$. 
Since $\p_{\ell}^{\top} \Pi_n \Theta \p_k = \p_{\ell}^{\top} \Pi_n \times \Theta \p_k$, 
this and Proposition \ref{prop:beta654} yield, for $n\ge m_0$,
\[
\begin{aligned}
\p_{\ell}^{\top} \Pi_n \Theta \p_k
&= \sum_{{\mu}=1}^{K} \sum_{i=1}^{m_{\mu}}
\left\{
\sum_{q=1}^{i} \binom{\ell}{q-1}p_{\mu}^{\ell-q+1} \binom{n}{i-q}p_{\mu}^{n-i+q}I_d
\right\}
\left\{
\sum_{j=i}^{m_{\mu}} \binom{k}{j-i}p_{\mu}^{k+i-j}\theta_{{\mu},j}
\right\}\\
&=\sum_{{\mu}=1}^{K} \sum_{j=1}^{m_{\mu}}
\left\{
\sum_{i=1}^{j} \binom{k}{j-i} \sum_{q=1}^{i} \binom{\ell}{q-1}\binom{n}{i-q}
\right\} 
p_{\mu}^{n+\ell+k+1-j} \theta_{{\mu},j}\\
&=\sum_{{\mu}=1}^{K} \sum_{j=1}^{m_{\mu}} \binom{n+k+\ell}{j-1} p_{\mu}^{n+\ell+k+1-j} \theta_{{\mu},j}
=\beta_{n+k+\ell+1}^*,
\end{aligned}
\]
as desired.
\end{proof}

For $n\in\N\cup\{0\}$, we define 
$H_n: \{(V_{[-n,\infty)}^X)^{\bot}\}^d \to \{(V_{(-\infty,-1]}^X)^{\bot}\}^d$ by
\[
H_n x := (\mathcal{H}_n x^1, \dots, \mathcal{H}_n x^d)^{\top}, 
\qquad 
x=(x^1,\dots,x^d)^{\top}\in (V_{[-n,\infty)}^X)^{\bot},
\]
and 
$\tilde{H}_n: \{(V_{(-\infty,-1]}^X)^{\bot}\}^d \to \{(V_{[-n,\infty)}^X)^{\bot}\}^d$ 
by 
\[
\tilde{H}_n x := (\tilde{\mathcal{H}}_n x^1, \dots, \tilde{\mathcal{H}}_n x^d)^{\top}, 
\qquad 
x=(x^1,\dots,x^d)^{\top} \in \{(V_{(-\infty,-1]}^X)^{\bot}\}^d.
\]
Then, 
by Lemma 4.2 in \cite{IKP2}, we have, for $\{s_{\ell}\}\in \ell_{2+}^{d\times d}$, 
\begin{equation}
H_n
\left(\sum_{\ell=0}^{\infty} s_{\ell} \tilde{\varepsilon}_{n+\ell+1}\right)
=-\sum_{j=0}^{\infty}
\left(\sum_{\ell=0}^{\infty} s_{\ell} \beta_{n+j+\ell+1}^*\right)
\varepsilon_j, 
\qquad 
\tilde{H}_n
\left(\sum_{\ell=0}^{\infty} s_{\ell} \varepsilon_{\ell}\right)
=-\sum_{j=0}^{\infty}
\left(\sum_{\ell=0}^{\infty} s_{\ell} \beta_{n+j+\ell+1}\right)
\tilde{\varepsilon}_{n+j+1}.
\label{eq:proj666}
\end{equation}

\begin{proposition}\label{prop:G123}
For $n\ge m_0$ and $v \in \C^{dM\times d}$, 
\begin{align}
H_n \left( \sum_{\ell=0}^{\infty} (v^{\top} \overline{\p}_{\ell}) \tilde{\varepsilon}_{n+\ell+1}\right)  
&= - \sum_{j=0}^{\infty} (v^{\top} \Lambda^{\top} \Pi_n \Theta \p_j) \varepsilon_j,
\label{eq:Hodd123}\\
\tilde{H}_n \left( \sum_{\ell=0}^{\infty} (v^{\top} \p_{\ell}) \varepsilon_{\ell}\right) 
&= - \sum_{j=0}^{\infty} (v^{\top} \Lambda (\Pi_n \Theta)^* \overline{\p}_j) \tilde{\varepsilon}_{n+j+1}.
\label{eq:Heven123}
\end{align}
\end{proposition}

\begin{proof}[\bf Proof]
First, we see from Lemma \ref{lem:beta162} that, for $n\ge m_0$ and $j\in\N\cup\{0\}$,
\[
\sum_{\ell=0}^{\infty} v^{\top} \overline{\p}_{\ell} \beta_{n+j+\ell+1}^*
= v^{\top}
\left(\sum_{\ell=0}^{\infty} \overline{\p}_{\ell} \p_{\ell}^{\top}\right) \Pi_n \Theta \p_j
=v^{\top} \Lambda^{\top} \Pi_n \Theta \p_j.
\]
This and the first equality in (\ref{eq:proj666}) yield (\ref{eq:Hodd123}). 
Next, we see from Lemma \ref{lem:beta162} that, for $n\ge m_0$ and $j\in\N\cup\{0\}$,
\[
\sum_{\ell=0}^{\infty} v^{\top} \p_{\ell} \beta_{n+j+\ell+1}
= v^{\top}
\left(\sum_{\ell=0}^{\infty} \p_{\ell} \p_{\ell}^*\right) (\Pi_n \Theta)^* \overline{\p}_j
=v^{\top} \Lambda (\Pi_n \Theta)^* \overline{\p}_j.
\]
This and the second equality in (\ref{eq:proj666}) give (\ref{eq:Heven123}). 
\end{proof}

Here is a key lemma.

\begin{lemma}\label{lem:Ginv123}
For $n\ge m_0$, both 
$I_{dM} - \tilde{G}_n G_n$ and $I_{dM} - G_n \tilde{G}_n$
are invertible and we have 
$\sum_{k=0}^{\infty} (\tilde{G}_n G_n)^k = (I_{dM} - \tilde{G}_n G_n)^{-1}$ 
and 
$\sum_{k=0}^{\infty} (G_n \tilde{G}_n)^k = (I_{dM} - G_n \tilde{G}_n)^{-1}$, 
where $(\tilde{G}_n G_n)^0 = (G_n \tilde{G}_n)^0 = I_{dM}$.
\end{lemma}

\begin{proof}[\bf Proof]
We assume $n\ge m_0$. 
It is enough for us to show that both 
$\sum_{k=0}^{\infty} (\tilde{G}_n G_n)^k$ and $\sum_{k=0}^{\infty} (G_n \tilde{G}_n)^k$ converge. 
We see from Proposition \ref{prop:G123} that, for $k\in\N$ and $v \in \C^{dM\times d}$, 
\[
(H_n \tilde{H}_n)^k \left( \sum_{\ell=0}^{\infty} (v^{\top} \p_{\ell}) \varepsilon_{\ell}\right)  
= \sum_{j=0}^{\infty} 
(v^{\top} \Lambda (\tilde{G}_n G_n)^{k-1} \tilde{G}_n \Pi_n \Theta \p_j) \varepsilon_{j},
\]
hence, for $k\in\N$ and $u, v\in \C^{dM\times d}$,
\[
\left\langle 
(H_n \tilde{H}_n)^k \left( \sum_{\ell=0}^{\infty} (v^{\top} \p_{\ell}) \varepsilon_{\ell}\right), 
\sum_{j=0}^{\infty} (u^{\top} \p_j) \varepsilon_{j}
\right\rangle 
=
v^{\top} \Lambda (\tilde{G}_n G_n)^{k-1} \tilde{G}_n \Pi_n \Theta 
\left(\sum_{j=0}^{\infty} \p_j \p_j^* \right) \overline{u}
= 
v^{\top} \Lambda (\tilde{G}_n G_n)^{k} \overline{u},
\]
and similarly for $k=0$. Since 
$(H_n \tilde{H}_n)^k x 
= ((\mathcal{H}_n \tilde{\mathcal{H}}_n)^k x^1, \dots, (\mathcal{H}_n \tilde{\mathcal{H}}_n)^k x^d)^{\top}$ 
for $x=(x^1,\dots,x^d)^{\top} \in \{(V_{(-\infty,-1]}^X)^{\bot}\}^d$, 
it follows from Proposition \ref{prop:Hbdd123} that
\[
\sum_{k=0}^{N} v^{\top} \Lambda (\tilde{G}_n G_n)^{k} \overline{u}
=\left\langle 
\sum_{k=0}^{N} 
(H_n \tilde{H}_n)^k \left( \sum_{\ell=0}^{\infty} (v^{\top} \p_{\ell}) \varepsilon_{\ell}\right), 
\sum_{j=0}^{\infty} (u^{\top} \p_j) \varepsilon_{j} \right\rangle
\]
converges as $N\to\infty$, for any $u, v\in \C^{dM\times d}$. 
By choosing $u_i, v_i\in \C^{dM\times d}\ (i=1,\dots,d)$ 
so that $(u_1,\dots,u_d) = (v_1,\dots,v_d)=I_{dM}$, 
we find that $\sum_{k=0}^{\infty} \Lambda (\tilde{G}_n G_n)^{k}$ 
converges. 
Since $\Lambda$ is invertible 
by Proposition \ref{prop:Lambda116}, 
$\sum_{k=0}^{\infty} (\tilde{G}_n G_n)^{k}$ also converges. 
Finally, from 
$\sum_{k=1}^{N} (G_n \tilde{G}_n)^k = G_n \left\{\sum_{k=0}^{N-1} (\tilde{G}_n G_n)^k\right\}\tilde{G}_n$ 
for $N\in\N$, 
$\sum_{k=0}^{\infty} (G_n \tilde{G}_n)^k$ converges, too.
\end{proof}

For $n\in\mathbb{N}$ and $k\in\mathbb{N}\cup\{0\}$,
the two sequences $\{b_{n,j}^k\}_{j=0}^{\infty}\in \ell_{2+}^{d\times d}$ and 
$\{\tilde{b}_{n,j}^k\}_{j=0}^{\infty}\in \ell_{2+}^{d\times d}$ 
are defined by the recursions
\[
b_{n,j}^0=\delta_{0,j}I_d,
\qquad 
b_{n,j}^{2k+1}
=\sum_{\ell=0}^{\infty} b_{n,\ell}^{2k} \beta_{n+j+\ell+1},
\qquad 
b_{n,j}^{2k+2}
=\sum_{\ell=0}^{\infty} b_{n,\ell}^{2k+1} \beta_{n+j+\ell+1}^*
\]
and 
\[
\tilde{b}_{n,j}^0=\delta_{0,j}I_d,
\qquad 
\tilde{b}_{n,j}^{2k+1}
=\sum_{\ell=0}^\infty
\tilde{b}_{n,\ell}^{2k} \beta_{n+j+\ell+1}^*,
\qquad 
\tilde{b}_{n,j}^{2k+2}
=\sum_{\ell=0}^{\infty}
\tilde{b}_{n,\ell}^{2k+1} \beta_{n+j+\ell+1},
\]
respectively (see Section 4 in \cite{IKP2}).

\begin{lemma}\label{lem:b146}
For $n\ge \max(m_0, 1)$, $k\in\mathbb{N}$ and $j\in\mathbb{N}\cup\{0\}$, we have
\begin{align}
b_{n,j}^{2k-1} &= \p_0^{\top} (\tilde{G}_n G_n)^{k-1} (\Pi_n \Theta)^* \overline{\p}_j,
\label{eq:bodd123}\\
b_{n,j}^{2k} &= \p_0^{\top} (\tilde{G}_n G_n)^{k-1} \tilde{G}_n \Pi_n \Theta \p_j,
\label{eq:beven123}\\
\tilde{b}_{n,j}^{2k-1} &= \p_0^{\top} (G_n \tilde{G}_n)^{k-1} \Pi_n \Theta \p_j,
\label{eq:tildebodd123}\\
\tilde{b}_{n,j}^{2k} &= \p_0^{\top} (G_n \tilde{G}_n)^{k-1} G_n (\Pi_n \Theta)^* \overline{\p}_j.
\label{eq:tildebeven123}
\end{align}
\end{lemma}

\begin{proof}[\bf Proof]
We assume $n\ge \max(m_0, 1)$, and prove (\ref{eq:bodd123}) and (\ref{eq:beven123}) by induction. 
First, from Lemma \ref{lem:beta162}, 
$b_{n,j}^1 = \beta_{n+j+1} = \p_0^{\top} (\Pi_n \Theta)^* \overline{\p}_j$. 
Next, for $k\in\N$, we assume (\ref{eq:bodd123}). Then, by Lemma \ref{lem:beta162}, 
\[
\begin{aligned}
b_{n,j}^{2k}
&= \sum_{\ell=0}^{\infty} b_{n,\ell}^{2k-1} \beta_{n+j+\ell+1}^*
= \sum_{\ell=0}^{\infty} \p_0^{\top} (\tilde{G}_n G_n)^{k-1} (\Pi_n \Theta)^* \overline{\p}_{\ell}
\p_{\ell}^{\top} \Pi_n \Theta \p_j\\
&= \p_0^{\top} (\tilde{G}_n G_n)^{k-1} (\Pi_n \Theta)^* 
\left(\sum_{\ell=0}^{\infty} \overline{\p}_{\ell} \p_{\ell}^{\top}\right) \Pi_n \Theta \p_j
=\p_0^{\top} (\tilde{G}_n G_n)^{k-1} (\Pi_n \Theta)^*  \Lambda^{\top} \Pi_n \Theta \p_j\\
&=\p_0^{\top} (\tilde{G}_n G_n)^{k-1} \tilde{G}_n \Pi_n \Theta \p_j
\end{aligned}
\]
or (\ref{eq:beven123}). 
From this as well as Lemma \ref{lem:beta162}, 
\[
\begin{aligned}
b_{n,j}^{2k+1}
&= \sum_{\ell=0}^{\infty} b_{n,\ell}^{2k} \beta_{n+j+\ell+1}
= \p_0^{\top} (\tilde{G}_n G_n)^{k-1} \tilde{G}_n \Pi_n \Theta 
\left(\sum_{\ell=0}^{\infty} \p_{\ell} \p_{\ell}^*\right) (\Pi_n \Theta)^* \overline{\p}_j\\
&=\p_0^{\top} (\tilde{G}_n G_n)^{k-1} \tilde{G}_n \Pi_n \Theta \Lambda (\Pi_n \Theta)^* 
\overline{\p}_j
=\p_0^{\top} (\tilde{G}_n G_n)^{k} (\Pi_n \Theta)^* \overline{\p}_j
\end{aligned}
\]
or (\ref{eq:bodd123}) with $k$ replaced by $k+1$. 
Thus (\ref{eq:bodd123}) and (\ref{eq:beven123}) follow. 
We can prove (\ref{eq:tildebodd123}) and (\ref{eq:tildebeven123}) by induction similarly; we omit the details. 
\end{proof}

We are now ready to prove Theorem \ref{thm:phirep116}.

\begin{proof}[\bf Proof of Theorem \ref{thm:phirep116}]
By Theorem 5.4 in \cite{IKP2}, we have 
$\phi_{n,j}
=\sum_{k=0}^\infty\{\phi_{n,j}^{2k}+\phi_{n,n-j+1}^{2k+1}\}$ 
for $n \in \N$, $j \in \{1,\dots,n\}$, 
where 
$\phi_{n,j}^{2k} := c_0 \sum_{\ell=0}^\infty b_{n,\ell}^{2k} a_{j+\ell}$ and 
$\phi_{n,j}^{2k+1} := c_0 \sum_{\ell=0}^\infty b_{n,\ell}^{2k+1} \tilde a_{j+\ell}$ 
for $n\in\mathbb{N}$ and $k, j\in\mathbb{N}\cup\{0\}$. 
Since $b_{n,j}^0 = \delta_{0,j}I_d$, we have $\phi_{n,j}^{0} = c_0 a_j$, 
$\phi_{n,j}
=c_0 a_j + \sum_{k=1}^\infty\{\phi_{n,j}^{2k}+\phi_{n,n-j+1}^{2k-1}\}$. 
By Lemma \ref{lem:b146}, we have, for $n\ge \max(m_0, 1)$, $k\in\mathbb{N}$ 
and $j \in \{1,\dots,n\}$, 
\begin{align*}
&\phi_{n,j}^{2k} 
= c_0 \p_0^{\top} (\tilde{G}_n G_n)^{k-1} \tilde{G}_n \Pi_n \Theta v_j
= c_0 \p_0^{\top} (\tilde{G}_n G_n)^{k-1} (\Pi_n \Theta)^* \Lambda^{\top} \Pi_n \Theta v_j,\\
&\phi_{n,n-j+1}^{2k-1} 
= c_0 \p_0^{\top} (\tilde{G}_n G_n)^{k-1} (\Pi_n \Theta)^* \tilde{v}_{n-j+1}.
\end{align*}
Therefore, thanks to Lemma \ref{lem:Ginv123}, we obtain the theorem.
\end{proof}


\section*{Acknowledgments}
The author would like to thank the Editor and referees for their helpful comments.


\end{document}